\documentclass{amsart}[12pt]

\newtheorem{theorem}{Theorem}[section]

\theoremstyle{definition}

\newtheorem{example}[theorem]{Example}

\theoremstyle{remark}

\newtheorem{Definition}{\bf Definition}[section]
\newtheorem{Thm}[Definition]{\bf Theorem}
\newtheorem{Lem}[Definition]{\bf Lemma}

\newtheorem{Prop}[Definition]{\bf Proposition}
\newtheorem{Example}[Definition]{\bf Example}
\newtheorem{Note}[Definition]{\bf Note}

\numberwithin{equation}{section}

\reversemarginpar

\newcommand{\nn}{\nonumber}
\newcommand{\no}{\noindent}

\newcommand{\realpart}{\mathop{\rm Re}\nolimits}

\newcommand{\ba}{\begin{eqnarray}}
\newcommand{\ea}{\end{eqnarray}}

\newcommand{\ione}{\int_{0}^{1}}

\newcommand{\allR}{\mathbb{R}}
\newcommand{\allC}{\mathbb{C}}
\newcommand{\allN}{\mathbb{N}}

\newcommand{\Nzero}{\mathbb{N}\cup\{0\}}

\newcommand{\stoli}{}

\newcommand{\barzeta}{{\bar\zeta}}
\newcommand{\eps}{\varepsilon}
\DeclareMathOperator{\Cl}{Cl}

\begin{document}

\title[The evaluation of Tornheim double sums. Part 2] {The evaluation of
Tornheim double sums. Part 2}

\author{Olivier Espinosa}
\address{Departamento de F{\'\i}sica,
Universidad T{\'{e}}c. Federico Santa Mar{\'\i}a, Valpara{\'\i}so, Chile}
\email{olivier.espinosa@usm.cl}

\author{Victor H. Moll}
\address{Department of Mathematics,
Tulane University, New Orleans, LA 70118}
\email{vhm@math.tulane.edu}

\subjclass{Primary 33}

\date{\today}


\begin{abstract}
We provide an explicit formula for the 
Tornheim double series $T(a,0,c)$ in terms of
an integral involving the Hurwitz zeta function. For 
integer values of the parameters,
$a=m, \; c=n$, we show that in the most interesting 
case of {\em even} weight $N:=m+n$
the Tornheim
sum $T(m,0,n)$ can be expressed in terms of zeta values 
and the family of integrals
\begin{equation}
\ione  \log \Gamma(q) B_{k}(q) \Cl_{l+1} (2 \pi q)  \, dq,
\notag
\end{equation}
with $k+l = N$, where $B_{k}(q)$ is a 
Bernoulli polynomial and $\Cl_{l+1} (x)$ is
a Clausen function.
\end{abstract}

\maketitle

\section{Introduction} \label{S:intro}

The function
\begin{equation}
T(a,b,c)  : =  \sum_{r=1}^{\infty} \sum_{s=1}^{\infty} \frac{1}{r^{a}
\, s^{b} \, (r+s)^{c} },
\label{seriestz}
\end{equation}
\no
was introduced by Tornheim in \cite{tornheim1}.  For $a, \, b, \, c \in 
\mathbb{R}$, the series is convergent if 
\begin{equation}
a+ c > 1, \, b+c > 1, \mbox{ and } a+b+c > 2.
\label{cond-conv}
\end{equation}
\noindent

\medskip
In the case $(a,b,c)=(m,k,n)$, with 
$m, \, k, \, n \in \mathbb{N} \cup \{ 0 \}$, we define
the {\em weight} of the Tornheim sum  
$T(m,k,n)$ as the positive integer $N = m+k+n$.

We have previously derived \cite{espmoll3} an analytic expression for
the general Tornheim sum $T(a,b,c)$ in terms of integrals involving the Hurwitz
zeta function $\zeta(z,q)$, defined as the meromorphic extension to the whole
complex $z$-plane of the series
\begin{equation}
\zeta(z,q) := \sum_{n=0}^{\infty} \frac{1}{(n+q)^{z}}, 
\label{zeta-def}
\end{equation}
\noindent
which is defined for $\realpart{z} > 1$ and $ q \neq 0, \, -1, \, -2, \cdots$.
Our expression for $T(a,b,c)$, recalled later in Theorem \ref{main Thm part I},
is valid for $a, \, b, \, c \in\allR - \Nzero$, provided the
convergence conditions \eqref{cond-conv} are satisfied. Using this result 
and a limiting procedure, we derived similar formulas for $T(m,k,n)$, with
$m, \, k, \, n \in \mathbb{N}$.

\medskip

In this paper we derive a formula for the Tornheim sum $T(a,0,c)$, valid for
$a, \, c \in\allR - \allN$ with $a>2$ and $c>2$. A limiting procedure will 
then provide an
analytic expression for the sums $T(m,0,n)$, with $m, \, n \in\allN-\{1\}$.

\medskip

These expressions for the sums $T(m,0,n)$ are of practical interest. In fact, 
Huard et al \cite{huard}  established the relation
\begin{equation}
T(m,k,n) = \sum_{i=1}^{m} \binom{m+k-i-1}{m-i} T(i,0,N-i) + 
\sum_{i=1}^{k} \binom{m+k-i-1}{k-i} T(i,0,N-i),
\label{huard-1}
\end{equation}
\noindent
for the Tornheim sum of weight $N = m+k+n$. Therefore, it suffices to consider
sums of the form $T(m,0,n)$. The convergence of the series requires 
$n>1$ and $m+n > 2$, thus the sum $T(m,0,1)$ diverges.

\medskip

Introduce the spaces
\begin{multline}
\mathcal{Z}_{N} := 
\{ T(m,k,n): \, m, \, k, \, n \in \mathbb{N} \cup \{ 0 \} 
\text{ with } n + m \ge 2, \, k + n \ge 2 \\
\mbox{ and } N = m+k+n \ge 3 \},
\end{multline}
\noindent
and
\begin{equation}
\mathcal{Z}_{N}^{0} := 
\{ T(m,0,n) \in \mathcal{Z}_{N} \}.
\end{equation}

The following result is contained in formula (\ref{huard-1}):

\begin{Prop}
\label{prop-hu1}
Every sum in $\mathcal{Z}_{N}$ is a linear combination of terms in 
$\mathcal{Z}_{N}^{0}$ with coefficients in $\mathbb{N}$.
\end{Prop}

\begin{Note}
The sums $T(m,0,n)$ appearing in (\ref{huard-1}) can be written as
\begin{eqnarray}
T(m,0,n)  & = &  \sum_{r=1}^{\infty} \sum_{s=1}^{\infty} \frac{1}{r^{m} 
\, (r+s)^{n} } \label{tmidd0} \\
& = & \sum_{r_{1} > r_{2}} 
\frac{1}{r_{1}^{n} \, r_{2}^{m}}. 
\nonumber
\end{eqnarray}
\noindent
Therefore $T(m,0,n)$ is a special case of the multiple zeta value (= MZV)
\begin{equation}
\zeta(s_{1},s_{2},\ldots,s_{k}) = \sum_{r_{1}>r_{2} > \cdots > r_{k} > 0 }
\prod_{j=1}^{k} r_{j}^{-s_{j}}, 
\label{def-MZV}
\end{equation}
\noindent
namely\footnote{The double zeta function appearing here should not be 
confused with the Hurwitz zeta function in (\ref{zeta-def}).}, 
\begin{equation}
T(m,0,n) = \zeta(n,m).
\end{equation}
\noindent
There is a vast literature on MZV and the reader if referred to Chapter 3 of 
\cite{borw2} for an introduction to this topic.
\end{Note}

\begin{Note}
In the case of {\em odd} weight, \cite{huard} gives the relation
\begin{multline}
T(m,0,n) =  (-1)^{m} \sum_{j=0}^{\lfloor{ \frac{n-1}{2} \rfloor} } 
\binom{m+n-2j-1}{m-1} \zeta(2j) \zeta(m+n-2j) 
\label{huard-2} \\
 +  (-1)^{m} \sum_{j=0}^{\lfloor{ \frac{m}{2} \rfloor} } 
\binom{m+n-2j-1}{n-1} \zeta(2j) \zeta(m+n-2j) -\tfrac{1}{2}\zeta(m+n),
\end{multline}
valid for $m\ge 1$. Here 
\begin{equation}
\zeta(s) = \sum_{n=1}^{\infty} \frac{1}{n^{s}}, 
\end{equation}
\noindent
is the classical Riemann zeta function. This function has an analytic 
extension to $\mathbb{C}- \{ 1 \}$, the point $s=1$ being a simple pole. 
Recalling that $\zeta(0)=-1/2$, the previous result can also be restated
in the following terms:
\end{Note}

\begin{Prop}
\label{prop-hu2}
Assume the weight $N$ is odd. Then the sums in $\mathcal{Z}_{N}^{0}$  can be 
evaluated as linear combination of the products $\zeta(2j) \zeta(N-2j)$,
$j=\max \left\{ \lfloor \frac{m}{2} \rfloor, \lfloor \frac{n-1}{2} 
\rfloor \right\} $,
with integer coefficients.
\end{Prop}

The idea of \cite{huard} is to produce a linear system of  equations for the 
unknowns $T_{i} := T(i,0,N-i)$, for $1 \leq i \leq N-2$. This system has full 
rank in the case $N$ odd and its solution yields (\ref{huard-2}).  Their 
methods also produce analytic expressions for Tornheim sums 
of  small {\em even} weight, but they fail in general. In particular, for 
weight $N \geq 8$, the system of equation mentioned above is not of full rank.
Section \ref{S:examples} contains a discussion of this issue.

\medskip

For example, the class
$\mathcal{Z}_{4}^{0}$ contains the sums
\begin{equation}
T(1,0,3) = \frac{1}{4}\zeta(4)\qquad \text{and}\qquad T(2,0,2) = \frac{3}{4} \zeta(4), 
\label{huard-4}
\end{equation}
\noindent 
and in the space $\mathcal{Z}_{6}^{0}$ we find the four sums 
\begin{align}\label{huard-6}
T(1,0,5) & =  - \frac{1}{2}\zeta(3)^{2} + \frac{3}{4} \zeta(6), \cr
T(2,0,4) & =  \zeta^{2}(3) - \frac{4}{3} \zeta(6), \cr
T(3,0,3) & =  \frac{1}{2}\zeta(3)^{2} - \frac{1}{2}\zeta(6), \cr
T(4,0,2) & = -\zeta^{2}(3) + \frac{25}{12} \zeta(6).
\end{align}

Huard et al. \cite{huard} also gave the relation 
\begin{equation}
\label{huard-7}
5T(2,0,6) + 2 T(3,0,5) = 10 \zeta(3) \zeta(5) - \frac{49}{4} \zeta(8),
\end{equation}
\noindent
for the case of weight $8$, but they are unable to 
evaluate the individual terms $T(2,0,6)$ and $T(3,0,5)$.  The question of 
an analytic expression these sums  remains open. 

\medskip

In this paper we give particular consideration to the Tornheim sums
$T(m,0,n)$ of arbitrary even weight $N = m+n$. For each even $N$, only 
two of these sums have known closed
expressions in terms of zeta values, namely $T(1,0,N-1)$ and $T(N/2,0,N/2)$.

Tornheim established the result
\begin{equation}
T(0,0,N) = \zeta(N-1)-\zeta(N),\quad N\ge 3,
\label{torn-two-0}
\end{equation}
which appears as Theorem 5, page 308 of Tornheim \cite{tornheim1}, 
\noindent
and the companion formula
\begin{equation}
T(1,0,N-1) = \frac{1}{2} \left[ (N-1) \zeta(N) - 
\sum_{i=2}^{N-2} \zeta(i) \zeta(N-i) \right], \quad N\ge 3,\label{tor-new3}
\end{equation}
\noindent
where $N=n+1$, can also be found in \cite{tornheim1}.

\medskip

The main result of this paper is an analytic expression for the 
Tornheim sums
$T(m,0,n)$ of {\em even} weight, with $m, \, n \geq 2$, in terms of a family 
of integrals involving
the {\em log-gamma} function $\log\Gamma(q)$, the 
{\em Bernoulli polynomials} $B_{k}(q)$, 
given by the generating function
\begin{equation}
\frac{e^{qt}}{e^{t}-1} = \sum_{k=0}^{\infty} B_{k}(q) \frac{t^{k-1}}{k!},
\label{B-def}
\end{equation}
and the {\em Clausen functions}  $\Cl_{l}(x)$, defined as
\begin{align}
\Cl_{2n} (x) & := 
\sum_{k=1}^{\infty}\frac{\sin kx}{k^{2n}},\quad n\in\allN, \label{def-clausen-even} \\
\Cl_{2n+1} (x) & :=
\sum_{k=1}^{\infty}\frac{\cos kx}{k^{2n+1}},\quad n\in\Nzero. \label{def-clausen-odd}
\end{align}

\medskip

\noindent
For example, we obtain 
\begin{equation}
T(6,0,2) = \tfrac{7}{6} \zeta(8)
- 6\zeta(3)\zeta(5) - Y_{2,6}^{*},
\label{nice-1}
\end{equation}
with
\begin{equation*}
Y_{2,6}^{*} := 
\frac{8}{3}\pi^{6} \left( X_{0,6} - 2 X_{1,5} + X_{2,4} \right) - 
6 \zeta(7) \log 2 \pi,
\end{equation*}
and where
\begin{equation}
X_{k,l} := (-1)^{\lfloor{l/2 \rfloor}} 
\frac{l!}{(2\pi)^l}\ione  \log \Gamma(q) B_{k}(q) 
\Cl_{l+1} (2 \pi q)  \, dq. 
\label{def-X1}
\end{equation}

%

\medskip

In the general case, we show that all the Tornheim sums of 
even weight $N$ can be expressed in terms of values of the Riemann zeta function
and integrals of the form
\noindent
\begin{multline}\label{def-Y-new}
Y_{m,N-m}^{*} := \frac{2(2\pi)^{N-2}}{m!(N-m-2)!} 
\sum_{j=0}^{m} (-1)^{j} \binom{m}{j} X_{j,N-2-j} \cr
+ (-1)^{\tfrac{N}{2} -1} \binom{N-2}{m-1} \zeta(N-1) \log 2 \pi,
\end{multline}
where $N$ is the weight and $m$ is even in the range 
$2 \leq m \leq 2 \left\lfloor \frac{N-2}{6} \right\rfloor$.

\medskip

The rest of the paper is organized as follows. Section \ref{S:main results}
contains all the main theorems we will prove in the subsequent sections.
Section \ref{S:T(a,0,c)} derives the expression for the Tornheim sum $T(a,0,c)$
in terms of an integral involving the Hurwitz zeta function, starting from a
more general result derived in \cite{espmoll3}.
Section \ref{S:T(m,0,n)} computes the limit of $T(a,0,c)$ as $a\to m$ and $c\to n$,
with $m,n\in\allN-\{1\}$. In Section \ref{S:N even - Clausen} we consider
the particular case of Tornheim sums of even weight $N$ and show that they
can be expressed in terms of zeta values and the family of integrals $X_{k,l}$,
defined above, with $k+l = N$. The explicit evaluation of this last family of
definite integrals remains a challenging problem. 
Finally, in Section \ref{S:examples} we give a systematic list of evaluations
for small even  weight.

\bigskip

\section{Main results} \label{S:main results}

We start by introducing some auxiliary special functions that play a role 
in our derivations. First, we have the {\em Bernoulli functions} 
\begin{equation}
A_{k}(q) = k \zeta'(1-k,q), \quad k\in\allN, \label{A-def}
\end{equation}
introduced in \cite{espmoll2,espmoll4}, and the kernel
\begin{equation}
K(q) = \log \sin \pi q \label{K-def}.
\end{equation}
\noindent

The first result is established in Section \ref{S:T(a,0,c)}.

\begin{Thm}
\label{Thm:T(a,0,c)}
Let $a,c \in \mathbb{R} - \mathbb{N}$ with $a, \, c > 2$. Then
\begin{multline}\label{T-explicit1}
T(a,0,c) =  4 \lambda(a) \lambda(c) \sin\left(\frac{\pi c}{2}\right)
\\
\times\Bigg[
\sin\left(\frac{\pi a}{2}\right)
\left\{
\zeta(1-a)\zeta(1-c)
-\frac{\zeta(1-a-c)B(a,c)}{1-\tan\left(\frac{\pi a}{2}\right)\tan\left(\frac{\pi c}{2}\right)}
\right\}
\\
- \frac{1}{2}\cos\left(\frac{\pi a}{2}\right)
\int_0^1\left[\zeta(1-a,q)-\zeta(1-a,1-q)\right]\zeta(1-c,q)\cot\pi q\,dq
\Bigg],
\end{multline}
\noindent
where
\begin{equation}
\lambda(z) := \frac{\Gamma(1-z)}{(2 \pi)^{1-z}} =
\frac{\pi}{(2 \pi)^{1-z} \, \Gamma(z) \, \sin \pi z}.
\label{def-lambda}
\end{equation}
\end{Thm}

\medskip

The explicit expression \eqref{T-explicit1} for 
$T(a,0,c)$ allows us to consider the
limit $a\to m, c\to n$, for $m,n\in\allN-\{1\}$.
The value of $T(m,0,n)$ is found to be given
in terms of the {\em basic integrals}:
\begin{align}
I_{BB}(k,l)  & : =  \int_{0}^{1} B_{k}(q) B_{l}(q) K(q) \, dq,
\label{basic} \\
I_{AB}(k,l)  & :=  \frac{1}{\pi}\int_{0}^{1} A_{k}(q) B_{l}(q) K(q) \, dq,
\nonumber   \\
I_{AA}(k,l) &  :=  \frac{1}{\pi^2}\int_{0}^{1} A_{k}(q) A_{l}(q) K(q) \, dq,
\nonumber  \\
J_{AA}(k,l) &  :=  \frac{1}{\pi^2}\int_{0}^{1} A_{k}(q) A_{l}(1-q) K(q) \, dq.
\nonumber 
\end{align}
where $B_k(q)$ is a Bernoulli polynomial and $A_k(q)$ and $K(q)$ are the functions
introduced in \eqref{A-def} and \eqref{K-def}, respectively.
\begin{Thm}
\label{Thm:T(m,0,n)}
Let $m\ge 2, \, n\ge 2 \in \mathbb{N}$. The Tornheim sum $T(m,0,n)$ is given 
by
\begin{equation}\label{T(m,0,n)-explicit-1}
T(m,0,n) = \zeta(m) \zeta(n) - \tfrac{1}{2} \zeta(m+n) +
(-1)^{\lfloor{\frac{m+n}{2} \rfloor}}\frac{(2\pi)^{m+n-1}}{m!n!}\ell_2(m,n),
\end{equation}
\noindent
where:
\begin{align}
\intertext{(a) $m$ and $n$ even:}
\ell_2 (m,n) &= m I_{AB}(m-1,n) + n I_{AB}(m,n-1),
\label{ell2 m even n even}
\\
\intertext{(b) $m$ and $n$ odd:}
\ell_2(m,n) &=  m I_{AB}(n,m-1) + n I_{AB}(n-1,m).
\label{ell2 m odd n odd}
\\
\intertext{(c) $m$ odd and $n$ even:}
\ell_2(m,n) &= \frac{1}{2} \left( m I_{BB}(m-1,n) + n I_{BB}(m,n-1)\right),
\label{ell2 m odd n even}
\\
\intertext{(d) $m$ even and $n$ odd:}
\ell_2(m,n) &= - m I_{AA}(m-1,n) - n I_{AA}(m, n-1)
\label{ell2 m even n odd}\\
\nn
&\phantom{=} - m J_{AA}(m-1,n)+ n J_{AA}(m,n-1), 
\end{align}    
\end{Thm}
The proof of this theorem is the subject of Section \ref{S:T(m,0,n)}.

\medskip

In the case of even weight $m+n = N$, the Tornheim sum $T(m,0,n)$ is given
in terms of the single family of integrals $I_{AB}(k,l)$. The next theorem 
shows that these 
can be expressed in terms of the family of integrals
\begin{equation}
X_{k,l} := (-1)^{\lfloor{l/2 \rfloor}} 
\frac{l!}{(2\pi)^l}\ione  \log \Gamma(q) B_{k}(q) 
\Cl_{l+1} (2 \pi q)  \, dq. 
\label{def-X}
\end{equation}

\begin{Thm}
\label{really-final}
Assume $m, \, n \in \mathbb{N}$ satisfy $m, \, n \geq 2$ and that the 
weight $N:= m+n$ is even. 
Define
\begin{equation}
T_1(m,n) := \zeta(m)\zeta(n) - \frac{1}{2} \zeta(N)
\end{equation}
and
\begin{equation}
T_2(m,n) := -\sum _{k=1}^{N/2-2} \binom{N-2 -2 k}{m-1}\zeta (2 k+1) \zeta (N-1-2k)
+(-1)^{N/2-1} Y_{m,n}^*,
\end{equation}
where
\begin{equation}\label{def-Y}
\begin{split}
Y_{m,n}^{*} &:= \frac{2(2\pi)^{N-2}}{m!(N-m-2)!} 
\sum_{j=0}^{m} (-1)^{j} \binom{m}{j} X_{j,N-2-j} \cr
&\qquad\qquad\qquad\qquad\qquad\qquad\qquad
+ (-1)^{\tfrac{N}{2} -1} \binom{N-2}{m-1} \zeta(N-1) \log 2 \pi \cr
Y_{m,n}^{*} &:= Y_{m,n}
+ (-1)^{\tfrac{N}{2} -1} \binom{N-2}{m-1} \zeta(N-1) \log 2 \pi.
\end{split}
\end{equation}

\medskip
\noindent
Then $T(m,0,n)$ can be written as
\begin{equation}
T(m,0,n) = 
\begin{cases}
T_1(m,n)+T_2(m,n), & \text{$m$ and $n$ odd}, \\
T_1(m,n)+T_2(n,m), & \text{$m$ and $n$ even}.
\end{cases}
\end{equation}

\medskip

It follows that, for even weight $N=m+n$, $T(m,0,n)$ is determined by
either $Y_{m,n}$ or $Y_{n,m}$, depending on the parities of $m$
and $n$. Note that both $Y_{m,n}$ and $Y_{n,m}$ are linear combinations
of integrals $X_{k,l}$ with fixed $k+l=N-2=m+n-2$.
\end{Thm}
We discuss this theorem in Section \ref{S:N even - Clausen}.

\bigskip

\section{An expression for the Tornheim series $T(a,0,c)$} 
\label{S:T(a,0,c)}

In this section we present an analytic expression for the Tornheim double series
$T(a,0,c)$, valid for $a>2, \, c>2$ and $a, \, c \not \in \mathbb{N}$. Note 
that 
$T(a,0,c)$ will 
be finite if  $a> 1$ and $c>1$. We employ here and in Section \ref{S:T(m,0,n)}
the shorthand notation
\begin{equation}
\bar{\zeta}(z,q) := \zeta(1-z,q)\quad \text{and} \quad
\bar{\zeta}(z) := \zeta(1-z), \quad\text{for }z\neq 0.
\end{equation}
\noindent
All the properties of the Hurwitz zeta function that we will use can be
expressed in terms of the function $\barzeta(z,q)$ as follows:
\begin{align}
\frac{d}{dq} \barzeta(z,q) &= (z-1)\barzeta(z-1,q),\\
\barzeta (k,q) &= - \frac{1}{k} B_k (q),\\
\barzeta'(k,q) &= - \frac{1}{k} A_k (q).
\end{align}
where $B_k(q)$ is a Bernoulli polynomial and $A_k(q)$ is the Bernoulli function
\eqref{A-def}.

The identities derived below appear from integration by parts. The restrictions
imposed on the parameters guaranteee that the boundary terms vanish. Recall
that the Hurwitz zeta function satisfies the identity
\begin{equation}
\nn
\zeta(z,q)=q^{-z}+\zeta(z,1+q),
\end{equation}
which implies $\zeta(z,0)=\zeta(z,1)$ if $z<0$. Equivalently,
\begin{equation}
\label{bt-zeta}
\barzeta(z,0)=\barzeta(z,1)\quad \text{if} \quad z>1.
\end{equation}

\begin{Thm}
\label{Thm:T(a,0,c)-b}
Let $a,c \in \mathbb{R} - \mathbb{N}$ with $a, \, c > 2$. Then
\begin{multline}\label{T-explicit0}
T(a,0,c) =  4 \lambda(a) \lambda(c) \sin\left(\frac{\pi c}{2}\right)
\\
\times\Bigg[
\sin\left(\frac{\pi a}{2}\right)
\left\{
\barzeta(a)\barzeta(c)
-\frac{\barzeta(a+c)B(a,c)}{1-\tan\left(\frac{\pi a}{2}\right)\tan\left(\frac{\pi c}{2}\right)}
\right\}
\\
- \frac{1}{2}\cos\left(\frac{\pi a}{2}\right)
\int_0^1\left[\barzeta(a,q)-\barzeta(a,1-q)\right]\barzeta(c,q)\cot\pi q\,dq
\Bigg].
\end{multline}
\end{Thm}

The proof is obtained by analyzing the behavior as $\eps \to 0$ of $T(a,b,c)$
given in (\ref{T-explicit}) below. This was first derived in \cite{espmoll3}.
The parameter $b$ is changed to $\eps$ in order to remind ourselves 
that it is small.

\begin{Thm}
\label{main Thm part I}
Let $a,b,c \in \mathbb{R}$, satisfying $a+c > 1, \, b+c >1$, and 
$a+b+c > 2$,  and define the auxiliary function $\lambda(z)$ as
in \eqref{def-lambda},
\ba
\lambda(z) & := & \frac{\Gamma(1-z)}{(2 \pi)^{1-z}} =
\frac{\pi}{(2 \pi)^{1-z} \, \Gamma(z) \, \sin \pi z}.
\ea
\no
For $a,b,c \not \in \mathbb{N}$ we have
\begin{equation}\label{T-explicit}
T(a,b,c) =  4 \lambda(a) \lambda(b) \lambda(c) \sin \left( \frac{\pi c}{2} 
\right)  Q(a,b,c) 
\end{equation}
\no
where
\begin{equation}
\begin{split}
Q(a,b,c) 
:= &\cos\left(\tfrac{\pi}{2}(a-b)\right) \left[ J(c,a,b)+J(c,b,a) \right]  \\
- &\cos\left(\tfrac{\pi}{2}(a+b)\right) \left[ I(a,b,c)+J(a,b,c) \right]
\end{split}
\end{equation}
\noindent
and $I(a,b,c), \, J(a,b,c)$ are the integrals defined by
\begin{align} 
I(a,b,c) &:= \ione \zeta(1-a,q) \zeta(1-b,q) \zeta(1-c,q) \, dq
\label{integral-I}
\intertext{and}
J(a,b,c) &:= \ione \zeta(1-a,q) \zeta(1-b,q) \zeta(1-c,1-q) \, dq,
\label{integral-J}
\end{align}
\noindent
where $\zeta(z,q)$ is the Hurwitz zeta function.
\end{Thm}

\medskip
\no
We consider the behavior of each of the factors 
in \eqref{T-explicit} as $b=\eps\to 0$. 

\medskip
\noindent
1) The term $\lambda(\eps)$ is regular in view of  $\lambda(0) = \frac{1}{2 \pi}$. 
The first few terms of its expansion are
\begin{equation}
\lambda(\eps) =  \frac{1}{2 \pi} + \frac{(\gamma + \log 2 \pi)}{2 \pi} \eps +
O( \eps^{2}). 
\end{equation}

\medskip

\no
2) The expansion
\begin{equation}
\cos \left( \frac{\pi}{2} ( a \pm \eps) \right) =  \cos \frac{\pi a}{2}  \mp
\frac{\pi}{2} \sin \frac{\pi a }{2} \eps  + O(\eps^{2})
\end{equation}
\no
is elementary.

\medskip

\no
3) To obtain the limiting value of  $Q(a, \eps, c)$ as $\eps \to 0$, it is 
required to analyze the 
behavior of $\barzeta(\eps, q)$ as it appears in
\begin{equation}
I(a,\eps,c) = \ione \barzeta(a,q) \barzeta(\eps, q) \barzeta(c,q) \, dq.
\label{int-25}
\end{equation}
\no
The function $\bar{\zeta}(\eps,q)$ has a pole at $\eps=0$ and its Laurent 
expansion at that pole is
\begin{equation}
\barzeta(\eps,q) = - \frac{1}{\eps} - \psi(q) + O(\eps).
\end{equation}
\no
where $\psi(q) = \frac{d}{dq}  \log \Gamma(q)$.
This expansion is not uniform in $q$, as the difference 
\begin{equation}
\barzeta(\eps,q) - \left( - 1/\eps - \psi(q) \right)
\label{blow}
\end{equation}
blows up as $q \to 0$, for any
{\em fixed } $\eps > 0$. Indeed,
\begin{equation}
\psi(q) = - \frac{1}{q} - \gamma +  O(q),
\end{equation}
\noindent
and 
\begin{equation}
\barzeta(\eps,q) = q^{\eps -1} + \zeta(1- \eps) + O(q),
\end{equation}
\noindent
showing that (\ref{blow}) is not bounded as $q \to 0$. 
This issue is resolved by shifting the second 
argument of the integral in (\ref{int-25}). Lemma \ref{recursion1} gives 
the integral $I(a, \eps, c)$
in terms of integrals of the form $I(a',1+\eps, c')$. The integrand here 
contains the factor $\barzeta(1+\eps,q) = \zeta(- \eps, q)$, whose expansion involves 
$\log \Gamma(q)$, producing a similar blow up. Lemma \ref{exp-corr} shifts the 
second argument again, and now $I(a, \eps, c)$ is expressed in terms of 
$I(a', 2 + \eps, c')$. The integrand now contains $\barzeta(2+\eps,q) = \zeta(-1 - \eps,q)$ and 
a uniform expansion is finally achieved.  \\

\begin{Lem}
\label{recursion1}
For $\eps > 0$ and $a, \, c > 1$ we have 
\begin{equation}
I(a,\eps,c) = -\frac{1}{\eps} 
\left[ (a-1) I(a-1,\eps+1,c) + (c-1) I(a,\eps+1,c-1) \right].
\end{equation}
\no
Similarly,
\begin{equation}
J(a,\eps,c) = -\frac{1}{\eps} 
\left[ (a-1) J(a-1,\eps+1,c) - (c-1) J(a,\eps+1,c-1) \right].
\end{equation}
\end{Lem}
\begin{proof}
The identity
\begin{equation}
\barzeta(\eps,q) = \frac{1}{\eps} \frac{d}{dq} \barzeta(1+\eps,q) 
\end{equation}
\no
yields, for $\eps > 0$,
\begin{align*}
I(a,\eps,c) & = \frac{1}{\eps} \ione \barzeta(a,q) \barzeta(c,q) \frac{d}{dq}
\barzeta(1+\eps,q) \, dq \cr
 & = \frac{1}{\eps} \barzeta(a,q) \barzeta(c,q) \barzeta(1+\eps,q) \Big|_{0}^{1}  -\frac{1}{\eps} \ione \barzeta(1+\eps,q) 
\frac{d}{dq} \left[ \barzeta(a,q) \barzeta(c,q) \right] \, dq.
\end{align*}
\no
Now observe that the boundary terms vanish for $\eps >0$ and 
$a, \, c > 1$, according to the identity \eqref{bt-zeta}.  The 
identity for $J$ follows along the same lines. 
\end{proof}

\medskip
\noindent
The expansion of $I(a, \eps, c)$ requires to iterate the result of Lemma 
\ref{recursion1} twice. 

\begin{Lem}
\label{exp-corr}
For $\eps > 0$ and $a, \, c > 2$, 
\begin{multline*}
I(a, \eps, c)  = \frac{1}{\eps(\eps+1)} \big[(a-1)(a-2) I(a-2,\eps+2,c)\\ 
+ 2(a-1)(c-1) I(a-1,\eps+2,c-1) + (c-1)(c-2) I(a,\eps+2,c-2)\big],
\end{multline*}
with similar expressions for $J(a, \eps, c)$ and $J(c, a, \eps)$.
\end{Lem}

Replacing in the expression (\ref{T-explicit}) gives 
\begin{equation}
\label{T-exp1}
T(a,\eps,c) = \frac{4 \lambda(a) \lambda(\eps) \lambda(c)}{\eps(\eps+1)} 
\sin \left( \frac{ \pi c}{2} \right) LT(a,\eps,c)
\end{equation}
\noindent
where 
\begin{equation}
LT(a,\eps,c) =  
\cos \left( \frac{\pi}{2}(a-\eps) \right) H_{+}(a,\eps,c) - \cos \left( \frac{\pi}{2} (a+\eps) \right) H_{-}(a,\eps,c),
\label{LT-def}
\end{equation}
\noindent
with
\begin{equation}
\label{H+ def}
\begin{split}
H_{+}(a,\eps,c)  &=
(1-a)(2-a)J(c,a-2,\eps+2) + 2(1-a)(1-c)J(c-1,a-1,\eps+2) \cr
&+ (1-c)(2-c)J(c-2,a,\eps+2) +  (1-c)(2-c)J(c-2,\eps+2,a) \cr
&- 2(1-a)(1-c) J(c-1,\eps+2,a-1) + (1-a)(2-a)J(c,\eps+2,a-2),
\end{split}
\end{equation}
\noindent
and
\begin{equation}
\label{H- def}
\begin{split}
H_{-}(a,\eps,c) &= (1-a)(2-a)I(a-2,\eps+2,c) + 2(1-a)(1-c)I(a-1,\eps+2,c-1) \cr
&+ (1-c)(2-c)I(a,\eps+2,c-2) +  (1-a)(2-a)J(a-2,\eps+2,c) \cr
&- 2(1-a)(1-c) J(a-1,\eps+2,c-1)+(1-c)(2-c)J(a,\eps+2,c-2).
\end{split}
\end{equation}

The last factor in (\ref{T-exp1}) is now expanded in powers of $ \eps$ to obtain
\begin{equation}
LT(a,\eps,c) = C_{a} (M_{+}-M_{-}) + \eps \left[ C_{a}(N_{+}-N_{-}) 
+ \frac{\pi}{2} S_{a}(M_{+}+M_{-}) \right] + O(\eps^{2}),
\label{LT}
\end{equation}
\noindent
with a self-explanatory notation. For example $M_{+}$ is the limit 
as $\eps \to 0$ of the expression multiplying 
$\cos(\tfrac{\pi}{2}(a-\eps))$ in \eqref{LT-def}, $N_{+}$ is the term of 
order $\eps$ in the same factor and 
$C_{a} = \cos(\pi a/2)$.   We now
reduce all the terms in (\ref{LT}). 

\begin{Prop}
The identity $M_{+} = M_{-}$ holds.
Therefore,  the term $LT(a, \eps)$ is of order $\eps$, and 
we conclude that  $T(a, \eps, c)$, in (\ref{T-exp1}), 
is non-singular as $\eps \to 0$. 
\end{Prop}
\begin{proof}
The proof of this result begins with
\begin{eqnarray}
& & \label{A-form} \\
M_{+} & = & (1-a)(2-a)J(c,a-2,2) + 2(1-a)(1-c)J(c-1,a-1,2) \nonumber \\
 &+ & (1-c)(2-c)J(c-2,a,2) +
  +  (1-c)(2-c)J(c-2,2,a) \nonumber \\
  & - & 2(1-a)(1-c)J(c-1,2,a-1)+(1-a)(2-a)J(c,2,a-2), \nonumber 
  \end{eqnarray}
\noindent
and
\begin{eqnarray}
M_{-} & = & (1-a)(2-a)I(a-2,2,c) + 2(1-a)(1-c)I(a-1,2,c-1)  \nonumber \\
&+  & (1-c)(2-c)I(a,2,c-2) 
 +  (1-a)(2-a)J(a-2,2,c) \nonumber \\
 & - & 2(1-a)(1-c)J(a-1,2,c-1)+(1-c)(2-c)J(a,2,c-2). \nonumber 
  \end{eqnarray}

We now establish some relations for the integrals $I$ and $J$. These are 
then employed to show that 
$M_{+}=M_{-}$. 

\begin{Lem}
\label{lemma6}
The integrals $I$ and $J$ satisfy
\begin{equation}
J(z,z',2) = I(z,z',2) \mbox{ and } J(z,2,z') = J(z',2,z).
\end{equation}
\end{Lem}
\begin{proof}
Start with
\begin{equation}
J(z,z',2) = \int_{0}^{1} \barzeta(z,q) \barzeta(z',q) \barzeta(2,1-q) \, dq,
\end{equation}
\noindent
with $\barzeta(2,q)=\zeta(-1,q)$.
From $\zeta(1-k,q) = -B_{k}(q)/k$ and the symmetry of $B_{2}$ we obtain
\begin{equation}
\barzeta(2,1-q)=\zeta(-1,1-q) = \zeta(-1,q)=\barzeta(2,q). 
\end{equation}
\noindent
The first identity follows from there. A similar argument establishes the 
second one.
\end{proof}

A direct calculation using Lemma \ref{lemma6}, now shows 
that $M_{+} = M_{-}$ in (\ref{LT}).
\end{proof}

\medskip

\noindent
{\bf Calculation of $M_{+}$}.  The 
next step is to simplify the expression 
for $M_{+}$ given in (\ref{A-form}).

\begin{Prop}
Let $a>2, \, c>2 \in \mathbb{R}$. Then
\begin{equation}
M_{+} = 2 \left( \barzeta(a) \barzeta(c) - \frac{\barzeta(a+c) B(a,c)}{1 - \tan ( \pi a/2)  \tan ( \pi c/2) } \right).
\label{mplus}
\end{equation}
\end{Prop}
\begin{proof}
Start by producing an alternative form of the term
$J(c,a-2,2)$ given by
\begin{eqnarray}
J(c,a-2,2) & = &  \int_{0}^{1} \barzeta(c,q) \barzeta(a-2,q) \barzeta(2,1-q) \, dq \nonumber \\
 & = & \frac{1}{a-2} \int_{0}^{1} \barzeta(c,q) \barzeta(2,1-q) \frac{d}{dq} \barzeta(a-1,q) \, dq \nonumber \\
  & = & \frac{1}{2-a} \int_{0}^{1} \barzeta(a-1,q) \frac{d}{dq} \left[ \barzeta(c,q) \barzeta(2,1-q) \right] \, dq 
  \nonumber \\
 & = & \frac{1}{2-a} \int_{0}^{1} \barzeta(a-1,q) \left[ (c-1) 
\barzeta(c-1,q) \barzeta(2,1-q) \right. \nn\\
&& 
\phantom{\frac{1}{2-a} \int_{0}^{1} \barzeta(a-1,q)(c-1) 
\barzeta(c-1,q)}\left. + \barzeta(c,q) \barzeta(1,1-q) \right] \, dq.
\nonumber
\end{eqnarray}
\noindent
Now recall that
\begin{equation*}
\barzeta(1,q)=-\barzeta(1,1-q),
\end{equation*}
since $\barzeta(1,q) = -B_{1}(q) = \frac{1}{2} - q$,
and also 
$$
\barzeta(2,q)=\barzeta(2,1-q),
$$
as shown before.  Conclude that
\begin{equation}
(a-2) J(c,a-2) = (1-c) I(a-1,c-1,2) + I(a-1,c,1).
\end{equation}
\noindent
Reversing the order of $a$ and $c$ and considering the expression 
in (\ref{A-form}), the first 
three terms of $M_{+}$ in (\ref{A-form})  are reduced to  
 \begin{equation}
(1-a) I(a-1,c,1) + (1-c) I(c-1,a,1).
 \end{equation}
 \noindent
 A similar analysis gives an expression for the last three terms 
in (\ref{A-form}). This yields a 
formula for $M_{+}$ in terms of the integrals $I$ and $J$ where 
one of the parameters is $1$. 
 
 \begin{Lem}
 \label{form-A1}
 The term $M_{+}$ in (\ref{A-form}) is given by
 \begin{equation}
M_{+} = (1-a) I(a-1,c,1) + (1-c) I(c-1,a,1) + (1-a)J(a-1,1,c) -(1-c)J(a,1,c-1).
\nonumber
 \end{equation}
 \end{Lem}
 
 We now consider the integrals appearing in this expression for $M_{+}$. 
 
 \begin{Lem}
 The integral $J$ satisfies 
 \begin{equation}
 J(u,1,v) = -J(v,1,u),
 \end{equation}
 \end{Lem}
 \begin{proof}
 This comes directly from $\barzeta(1,1-q) = -\barzeta(1,q)$.
 \end{proof}
 
 The  expression for $M_{+}$ in Lemma \ref{form-A1} is now simplified.
 
 \begin{Lem}
 Let $a>2, \, c>2 \in \mathbb{R}$. Then
 \begin{equation}
 (1-a) I(a-1,c,1) + (1-c)I(a,c-1,1) = \barzeta(a) \barzeta(c) - \int_{0}^{1} \barzeta(a,q) \barzeta(c,q) \, dq
 \nonumber
 \end{equation}
 \noindent
 and 
 \begin{equation}
 (1-a) J(a-1,1,c) -(1-c)J(a,1,c-1) = \barzeta(a) \barzeta(c) - \int_{0}^{1} \barzeta(a,q) \barzeta(c,1-q) \, dq.
 \nonumber
 \end{equation}
 \noindent
 Therefore
 \begin{equation}
 \label{form-A3}
 M_{+} = 2 \barzeta(a)\barzeta(c) - \int_{0}^{1} \barzeta(a,q) \barzeta(c,q) \, dq- \int_{0}^{1} \barzeta(a,q) \barzeta(c,1-q) \, dq.
 \end{equation}
 \end{Lem}
 \begin{proof}
 Start with
 \begin{equation}
 (1-a) I(a-1,c,1) = (1-a) \int_{0}^{1} \barzeta(a-1,q) \barzeta(c,q) \barzeta(1,q) \, dq  \nonumber
 \end{equation}
 \noindent
 and  the identity $(a-1) \barzeta(a-1,q) = \frac{d}{dq} \barzeta(a,q)$. Integration by parts and the value $\barzeta(1,q) = \frac{1}{2} - q$ give the first identity. The rest of the formulas are established in a similar form.
 \end{proof}

The identites 
\begin{equation}
\int_{0}^{1} \zeta(1-a,q) \zeta(1-c,q) \, dq = \zeta(1-a-c) B(a,c) \frac{\cos( \tfrac{\pi}{2}(a-c))}{\cos( \tfrac{\pi}{2}(a+c))},
\nonumber
\end{equation}
\noindent
and
\begin{equation}
\int_{0}^{1} \zeta(1-a,q) \zeta(1-c,1-q) \, dq = \zeta(1-a-c) B(a,c) \nonumber
\end{equation}
\noindent
appear in Theorem $3.1$ in \cite{espmoll1}. Replacing in (\ref{form-A3}) 
produces the final expression for $M_{+}$ claimed in (\ref{mplus}). 
\end{proof}

\vskip 0.5in

\noindent
{\bf Calculation of $N_{+}$}. This is defined in (\ref{LT}) as the term of 
order $\eps$ in the expansion of $H_{+}(a,\eps,c)$ defined in \eqref{H+ def}.
We illustrate the general method of calculation by computing the term of 
order $\eps$ in the integral: 
\begin{equation*}
J(c,a-2,\eps+2) = \int_{0}^{1} \barzeta(c,q) \barzeta(a-2,q) \zeta(-1-\eps, 1-q) \,dq. 
\end{equation*}
\noindent
The term of order $\eps$ is 
\begin{multline*}
- \int_{0}^{1} \barzeta(c,q) \barzeta(a-2,q) \zeta'(-1,1-q) \, dq \\
= -\frac{1}{a-2} \int_{0}^{1} \left(\frac{d}{dq} \barzeta(a-1,q)\right) \barzeta(c,q) \zeta'(-1,1-q) \, dq.
\end{multline*}
\noindent
Integrate by parts and observe that the boundary terms vanish to see that 
the term of order $\eps$ is
\begin{equation*}
\frac{1}{a-2} \int_{0}^{1} \barzeta(a-1,q) \frac{d}{dq} \left[ \barzeta(c,q) 
\zeta'(-1,1-q) \right]\, dq.
\end{equation*}
\noindent
Thus, 
\begin{equation*}
(a-2) \times \mbox{ the term of order  }\eps \mbox{ in } J(a-2,c,\eps+2)
\end{equation*}
\noindent
is
\begin{multline*}
(c-1)\int_{0}^{1} \barzeta(c-1,q) \barzeta(a-1,q) \zeta'(-1,1-q) \, dq \\
- \int_{0}^{1} \barzeta(c,q) \barzeta(a-1,q)  
\frac{d}{du} \zeta'(-1,u) \Big{|}_{u=1-q} \, dq.
\end{multline*}
\noindent
Now use
\begin{equation*}
\frac{d}{du}  \zeta'(-1,u) = \frac{\partial}{\partial z} \Big{|}_{z=-1} \frac{\partial}{\partial u} \zeta(z,u) =  \frac{\partial}{\partial z} \Big{|}_{z=-1} ( -z \zeta(z+1,u) )  =  -\zeta(0,u) + \zeta'(0,u),
\end{equation*}
\noindent
to check that the contribution coming from $\zeta'(-1,1-q)$ vanishes and 
to verify the result described next. 

\begin{Lem}
The part of $N_{+}$ coming  from integrals with  
$\eps+2$ in the third variable  is
\begin{equation*}
\begin{split}
-(1-a)& \left[ \int_{0}^{1} \barzeta(c,q) \barzeta(a-1,q) \zeta(0,1-q) \, dq -  \int_{0}^{1} \barzeta(c,q) \barzeta(a-1,q) \zeta'(0,1-q) \, dq   \right] \cr
-(1-c)& \left[ \int_{0}^{1} \barzeta(a,q) \barzeta(c-1,q) \zeta(0,1-q) \, dq -  \int_{0}^{1} \barzeta(a,q) \barzeta(c-1,q) \zeta'(0,1-q) \, dq \right].
\end{split}
\end{equation*}
\end{Lem}

Similar calculations produce the other parts of $N_{+}$.  We spare the 
reader the details. 

\begin{Prop}
The term $N_{+}$ is given by 
\begin{equation*}
\begin{split}
 2 \barzeta(a) \barzeta(c) &- \int_{0}^{1} \barzeta(a,q) \barzeta(c,q) \, dq - \int_{0}^{1} \barzeta(a,1-q) \barzeta(c,q) \, dq  \cr
 &+ (1-a)\left[  \int_{0}^{1} \barzeta(c,q) \barzeta(a-1,q) \zeta'(0,1-q) \, dq \right.\cr
 & \phantom{(1-a)xxxxxxxx} \left. +\int_{0}^{1} \barzeta(c,1-q) \barzeta(a-1,q) \zeta'(0,1-q) \, dq  \right]  \cr
 &+ (1-c)\left[  \int_{0}^{1} \barzeta(a,q) \barzeta(c-1,q) \zeta'(0,1-q) \, dq \right.\cr
 & \phantom{(1-c)xxxxxxxx}\left.-\int_{0}^{1} \barzeta(a,1-q) \barzeta(c-1,q) \zeta'(0,q) \, dq \right].
\end{split}
\end{equation*}
\end{Prop}

The process is now repeated to produce a similar expression for $N_{-}$. The 
result is stated next. 

\begin{Prop}
\label{Prop-212}
The difference $N_{+} - N_{-}$ is given by 
\begin{align}
N_{+}-N_{-}& = (1-a) \int_{0}^{1} \barzeta(a-1,q) \barzeta(c,q) \left[ \zeta'(0,q) + \zeta'(0,1-q) \right] \, dq \nonumber \cr
& + (1-a) \int_{0}^{1} \barzeta(a-1,q) \barzeta(c,1-q) \left[ \zeta'(0,q) + \zeta'(0,1-q) \right] \, dq \nonumber \cr
& + (1-c) \int_{0}^{1} \barzeta(a,q) \barzeta(c-1,q) \left[ \zeta'(0,q) + \zeta'(0,1-q) \right] \, dq \nonumber \cr
& - (1-c) \int_{0}^{1} \barzeta(a,q) \barzeta(c-1,1-q) \left[ \zeta'(0,q) + \zeta'(0,1-q) \right] \, dq. \nonumber 
\end{align}
\end{Prop}

The difference $N_{+}-N_{-}$ is now given in the form stated in Theorem 
\ref{Thm:T(a,0,c)}. Recall that
that $\zeta'(0,q) = \log \Gamma(q) - \log \sqrt{2 \pi}$, so that 
\begin{equation}
\zeta'(0,q) + \zeta'(0,1-q) = \log \left[ \Gamma(q) \Gamma(1-q) \right] 
- \log 2 \pi   = - \log (2 \sin \pi q ), \label{gamma-term}
\end{equation}
\noindent
where we have used the reflection rule for the Gamma function
\begin{equation}
\Gamma(q) \Gamma(1-q) = \frac{\pi}{\sin \pi q}.
\end{equation}
\noindent
Integration by parts gives
\begin{align*}
(1-a) \int_{0}^{1} \barzeta(a-1,q) \barzeta(c,q) \, dq &= -(1-c) \int_{0}^{1} \barzeta(a,q) \barzeta(c-1,q) \, dq, \cr
\intertext{and} 
(1-a) \int_{0}^{1} \barzeta(a-1,q) \barzeta(c,1-q) \, dq &= (1-c) \int_{0}^{1} \barzeta(a,q) \barzeta(c-1,1-q) \, dq.
\end{align*}
\noindent
These identities show that in the expansion of $N_{+}-N_{-}$ in 
Proposition \ref{Prop-212}, we 
can ignore the contribution of the $q$-independent term 
$-\log 2$ in \eqref{gamma-term} and 
produce our final expression for the term $N_{+}-N_{-}$.

\begin{Prop}
The term $N_{+}-N_{-}$ in (\ref{LT}) is given by 
\begin{multline*}
 \int_{0}^{1} \frac{d}{dq} \left[ \barzeta(c,q) \left( \barzeta(a,q) - \barzeta(a,1-q) \right) \right]
\log \sin ( \pi q) \, dq  \cr
 =  - \pi \int_{0}^{1} \barzeta(c,q) \left[ \barzeta(a,q) - \barzeta(a,1-q) \right] \cot( \pi q) \, dq. 
\end{multline*}
\end{Prop}

Taking the limit as $\eps \to 0$ in (\ref{T-exp1}), we obtain 
the result described in Theorem  
\ref{Thm:T(a,0,c)}. 

\bigskip

\section{An expression for the Tornheim sum $T(m,0,n)$}
\label{S:T(m,0,n)}

In this section we analyze the behavior of 
$T(a,0,c)$ given in Theorem \ref{Thm:T(a,0,c)}
as the parameters $a$ and $c$ approach positive integer values.
The notation
$a = m + \varepsilon_{a}, \;
c = n + \varepsilon_{c}$
\no
with $m,n \in \mathbb{N},m,n\ge 2$ 
and $\varepsilon_{a}, \varepsilon_{a} \to 0$, is empoyed throughout. \\

Observe that $\lambda(z)$ is singular as $z$ becomes a positive 
integer, so in order to 
obtain the limiting value $T(m,0,n)$, it remains to consider 
the limiting behavior of $T(a,0,c)$.

Define
\begin{multline}
\label{L1-def}
L_1 (m,n): = \mathop {\lim }\limits_{a \to m} \mathop {\lim }\limits_{c \to n} 4\lambda (a)\lambda (c)
\sin \left( {\frac{{\pi a}}{2}} \right)\sin \left( {\frac{{\pi c}}{2}} \right)\\
\times\left\{ {\bar \zeta \left( a \right)\bar \zeta \left( c \right) - \frac{{\bar \zeta \left( {a + c} \right)B(a,c)}}
{{1 - \tan \left( {\frac{{\pi a}}{2}} \right)\tan \left( {\frac{{\pi c}}{2}} \right)}}} \right\}
\end{multline}
\no
and 
\begin{multline}
\label{L2-def}
L_2 (m,n): = -\mathop {\lim }\limits_{a \to m} \mathop {\lim }\limits_{c \to n} 
2\lambda (a)\lambda (c)\cos\left(\frac{\pi a}{2}\right)\sin\left(\frac{\pi c}{2}\right)\\
\times\int_0^1\left[\barzeta(a,q)-\barzeta(a,1-q)\right]\barzeta(c,q)\cot\pi q\,dq.
\end{multline}

\noindent
The expressions for $L_{1}(m,n)$ and $L_{2}(m,n)$ given in this section 
give Theorem \ref{Thm:T(m,0,n)} from Theorem \ref{Thm:T(a,0,c)}. \\

\noindent
{\bf The evaluation of $L_{1}(m,n)$.} This is elementary.  The simplification
employs the relations
\begin{align}
\zeta(2k) & = \frac{(-1)^{k+1} (2 \pi)^{2k} B_{2k}}{2(2k)!}, 
&& k \in \mathbb{N} \cup \{ 0 \}, \label{zeta-simp} \\
\zeta(1-k) & = \frac{(-1)^{k+1} B_{k}}{k},  
&& k \in \mathbb{N},  \nonumber \\
\intertext{and}
\zeta'(-2k) & = (-1)^{k} \frac{(2k)! \zeta(2k+1)}{2 (2 \pi)^{2k} }, 
&& k \in \mathbb{N}.
\nonumber
\end{align}

\begin{Prop}
\label{Lem:L1}
The function $L_{1}(m,n)$ defined in (\ref{L1-def}) is given by
\begin{equation}
L_1 (m,n) = \zeta(m) \zeta(n) - \frac{1}{2} \zeta(m+n).
\end{equation}
\end{Prop}
\begin{proof}
The expansion of $\lambda(z)$ about $z = k\in \mathbb{N}$ is 
\ba
\label{lambda-expansion}
\lambda \left( {k + \varepsilon } \right) = ( - 1)^k \frac{{\left( {2\pi } \right)^{k - 1} }}
{{\Gamma \left( k \right)}}\left[ {\frac{1}
{\varepsilon } + \log 2\pi  - \psi \left( k \right) + O\left( \varepsilon  \right)} \right].
\ea
\no
We now examine the behavior of the other factors in (\ref{L1-def}) as
$\varepsilon_a,\varepsilon_c$ tend to zero. The result depends on the 
parities of $m$ and $n$.\\

To simplify our notation we define the {\em $4$-parity} of the integer $k$,
$p_{k}$, as
\begin{equation}
p_{k} := (-1)^{\lfloor{k/2 \rfloor}} = \begin{cases}
               (-1)^{k/2}, & \text{$k$ even}, \\
               (-1)^{(k-1)/2}, & \text{$k$ odd}, 
           \end{cases}
\label{4-parity}
\end{equation}

\no
Then, for $k$ even:

\begin{align}
\label{sin-expansion-n-even}
\sin \left( {\frac{\pi }{2}(k + \varepsilon )} \right) &= 
\frac{\pi }{2}p_{k} 
\varepsilon  + O\left( {\varepsilon ^2 } \right), \\
\tan \left( {\frac{\pi }{2}(k + \varepsilon )} \right) &= 
\frac{\pi }{2}\varepsilon + O\left( {\varepsilon ^2 }\right),\\
    &  \nn \\
\barzeta \left( {k + \varepsilon } \right) &=  
\zeta \left( 1-k \right)  + O\left( {\varepsilon} \right),
\intertext{and for $k$ odd:}
\label{sin-expansion-n-odd}
\sin \left( {\frac{\pi }{2}(k + \varepsilon )} \right) &= 
p_{k} + O\left( {\varepsilon } \right),  \\
\tan \left( {\frac{\pi }{2}(k + \varepsilon )} \right) &= 
-\frac{2}{\pi \varepsilon} + O\left( {\varepsilon ^0 }\right),  \\
    &  \nn \\
\barzeta \left( {k + \varepsilon } \right) &=  
- \zeta '\left( 1-k \right)\varepsilon  + O\left( {\varepsilon ^2 } \right),
\qquad k\ge 3,
\label{zeta-expansion-1}
\end{align}
\no
since $\zeta(-2k) = 0$ for $k\in\mathbb{N}$. \\

\noindent
Therefore
\begin{align*}
\intertext{(a) $m$ and $n$ even:}
L_1 (m,n) &= \frac{{\left( {2\pi } \right)^{m + n } }}
{{4 \Gamma \left( m \right)\Gamma \left( n \right)}}
 \left[ {\zeta \left( {1 - m} \right)\zeta \left( {1 - n} \right)-B(m,n)\zeta \left( {1 - m - n} \right)} \right]p_{m} p_{n}
\\
\intertext{(b) $m$ even and $n$ odd:}
L_1 (m,n) &= \frac{{\left( {2\pi } \right)^{m + n - 1} }}
{{\Gamma \left( m \right)\Gamma \left( n \right)}}
 \left[ { \zeta \left( {1 - m} \right)\zeta '\left( {1 - n} \right)
 - B(m,n)\zeta '\left( {1 - m - n} \right)} \right]p_{m}p_{n}
 \\
\intertext{(c) $m$ odd and $n$ even:}
L_1 (m,n) &= \frac{{\left( {2\pi } \right)^{m + n - 1} }}
{{\Gamma \left( m \right)\Gamma \left( n \right)}}
 \left[ { \zeta '\left( {1 - m} \right)\zeta \left( {1 - n} \right)
 - B(m,n)\zeta '\left( {1 - m - n} \right)} \right]p_{m}p_{n}
\\
\intertext{(d) $m$ and $n$ odd:}
L_1 (m,n) &= \frac{{\left( {2\pi } \right)^{m + n } }}
{{4 \Gamma \left( m \right)\Gamma \left( n \right)}}
 \left[ \frac{4}{\pi^{2}} \zeta '\left( {1 - m} \right)\zeta '\left( {1 - n} \right) +
B(m,n)\zeta \left( {1 - m - n} \right) \right] p_{m}p_{n}.
\end{align*}    

The result now follows by using the relations (\ref{zeta-simp}).
\end{proof}

\noindent
{\bf The evaluation of $L_{2}(m,n)$}. This proceeds 
along similar lines. We employ the expansion
\ba
\nn
\barzeta (k + \varepsilon ,q) 
&=& \barzeta (k,q) + \varepsilon \barzeta '(k,q) + o(\varepsilon )
\label{zetabar-expansion} \\
 &=&  - \frac{1}{k}\left[ {B_k (q) + \varepsilon A_k (q)} \right] + 
o(\varepsilon ), \nn
\ea
\no
and the reflection property of the 
Bernoulli polynomials,
\begin{equation}
\label{Bernoulli-reflection}
B_k (1 - q) = ( - 1)^k B_k (q),
\end{equation}
\no
to establish the vanishing of some of the integrals that appear in
intermediate calculations. The results are expressed in terms of the 
four integrals $I_{AA}, \, I_{AB}, \, I_{BB}, \, J_{AA}$ introduced in 
(\ref{basic}).

\medskip

\begin{Prop}
\label{Lem:L2}
The limit $L_2(m,n)$ defined in (\ref{L2-def}) is given by
\begin{equation}
L_2 (m,n) = p_{m+n}\frac{(2\pi)^{m+n-1}}{m!n!}\ell_2(m,n).
\end{equation}
Here $p_{n+m}$ is defined in (\ref{4-parity}) and the 
function $\ell_2(m,n)$ is given in terms of the following basic integrals,
\begin{align}
I_{BB}(k,l)  & : =  \int_{0}^{1} B_{k}(q) B_{l}(q) K(q) \, dq,
\label{basic-1} \\
I_{AB}(k,l)  & :=  \frac{1}{\pi}\int_{0}^{1} A_{k}(q) B_{l}(q) K(q) \, dq,
\\
I_{AA}(k,l) &  :=  \frac{1}{\pi^2}\int_{0}^{1} A_{k}(q) A_{l}(q) K(q) \, dq,
\\
J_{AA}(k,l) &  :=  \frac{1}{\pi^2}\int_{0}^{1} A_{k}(q) A_{l}(1-q) K(q) \, dq.
\end{align}
by
\begin{align}
\intertext{(a) $m$ and $n$ even:}
\ell_2 (m,n) &= m I_{AB}(m-1,n) + n I_{AB}(m,n-1),
\label{ell2 m even n even-0}
\\
\intertext{(b) $m$ and $n$ odd:}
\ell_2(m,n) &=  m I_{AB}(n,m-1) + n I_{AB}(n-1,m).
\label{ell2 m odd n odd-0}
\\
\intertext{(c) $m$ odd and $n$ even:}
\ell_2(m,n) &= \frac{1}{2} \left( m I_{BB}(m-1,n) + n I_{BB}(m,n-1)\right),
\label{ell2 m odd n even-0}
\\
\intertext{(d) $m$ even and $n$ odd:}
\ell_2(m,n) &= - m I_{AA}(m-1,n) - n I_{AA}(m, n-1)
\label{ell2 m even n odd-0}\\
\nn
&\phantom{=} - m J_{AA}(m-1,n)+ n J_{AA}(m,n-1), 
\end{align}    

\end{Prop}
\begin{proof}
We make use of the expansion
\ba
\label{cos-expansion}
\cos \left( {\frac{\pi }{2}(k + \varepsilon )} \right)& = &
\begin{cases}
p_{k} +o(\varepsilon),\quad k \text{  even} \\
-p_{k} \displaystyle{\frac{\pi}{2}}\varepsilon+o(\varepsilon),\quad k \text{  odd}.
\end{cases}
\ea

In the case $m$ even, the singularity of 
$\lambda(m+\varepsilon_a)$ as $\varepsilon_a\to 0$ is balanced by
\begin{equation*}
\bar \zeta \left( {m + \varepsilon _a ,q} \right) - \bar \zeta \left( {m + \varepsilon _a ,1 - q} \right) =  - \frac{1}{m}\left[ {A_m (q) - A_m (1 - q)} \right]\varepsilon _a  + O\left( {\varepsilon _a }^2 \right).
\end{equation*}
For $m$ odd, the combination 
$\lambda(m+\varepsilon_a)\cos\left((m+\varepsilon_a)\pi/2\right)$
is regular as $\varepsilon_a\to 0$.\\
The singularity at $n+ \varepsilon_{c}$ 
is treated similarly: the term
$\lambda(n+\varepsilon_c)\sin\left((n+\varepsilon_c)\pi/2\right)$
is regular for $n$ even; for $n$ odd, the singularity of
$\lambda(n+\varepsilon_c)$ requires the vanishing of the
integral in (\ref{L2-def}). This is a consequence of the symmetry of the 
function
$\left[\barzeta(a,q)-\barzeta(a,1-q)\right]\cot\pi q$, about $q  = 1/2$. \\

Introduce the notation
\begin{equation}
\alpha_{m,n} := p_{m+n} \frac{(2 \pi)^{m+n-2}}{m! \, n!},
\end{equation}
\noindent
where $p_{m+n}$ is defined in (\ref{4-parity}). Then

\begin{align*}
\intertext{(a) for $m$ and $n$ even:}
L_2 (m,n) &=  - \pi  \alpha_{m,n} 
\int_0^1 {\left[ {A_m (q) - A_m (1 - q)} \right]} B_n (q)\cot \pi q\,dq,
\\
\intertext{(b) for $m$ even and $n$ odd:}
L_2 (m,n) &=  2  \alpha_{m,n}
\int_0^1 {\left[ {A_m (q) - A_m (1 - q)} \right]} A_n (q)\cot \pi q\,dq,
\\
\intertext{(c) for $m$ odd and $n$ even:}
L_2 (m,n) &=  - \pi ^2  \alpha_{m,n}
\int_0^1 {B_m (q)} B_n (q)\cot \pi q\,dq,
\\
\intertext{(d) for $m$ and $n$ odd:}
L_2 (m,n) &=  - 2\pi \alpha_{m,n} 
\int_0^1 {B_m (q)} A_n (q)\cot \pi q\,dq.
\end{align*}
The final expression in Proposition  \ref{Lem:L2} is obtained by 
writing the integrals above in terms of the basic integrals defined in 
(\ref{basic-1}). This 
can be achieved by using
\[
\pi\cot\pi q = \frac{d}{dq}\log\sin\pi q = K'(q),
\]
integrating by parts and employing the recurrence relations
\begin{align}
\frac{d}{dq}B_k(q) &= k B_{k-1}(q),
\label{B derivative}\\
\frac{d}{dq}A_k(q) &= k A_{k-1}(q)+\frac{k}{k-1}B_{k-1}(q),\quad k\ge2.
\label{A derivative}
\end{align}
See \cite{espmoll2} for a derivation of \eqref{A derivative}.

\medskip\noindent
The kernel $K(q)$ has only a logarithmic singularity at $q=0$,
so it is  possible to separate the integrals involving the difference
${A_k (q) - A_k (1 - q)}$. The
functions $A_k(q)$ are well behaved in $[0,1]$. Only
$A_1(q)=\log\left[\Gamma(q)/\sqrt{2\pi}\right]$ has a singularity at $q=0$,
and this is only logarithmic, and therefore integrable in the cases of 
interest here.
\end{proof}

\begin{example}
\label{T202-a}
The Tornheim sum $T(2,0,2)$ is the simplest example of even weight. Its 
exact value is
\begin{equation}
T(2,0,2) = \frac{3}{4} \zeta(4) = \frac{\pi^{4}}{120},
\end{equation}
\noindent
see (\ref{euler2-0}). Theorem
\ref{Thm:T(m,0,n)} states that $T(2,0,2) = L_{1}(2,2) + L_{2}(2,2)$, 
with
\begin{equation}
L_{1}(2,2) = \frac{\pi^{4}}{45},
\end{equation}
\noindent
and 
\begin{equation}
L_{2}(2,2) = 4 \pi^{3} \left( I_{AB}(1,2) + I_{AB}(2,1) \right).
\end{equation}
\noindent
Therefore
\begin{equation}
T(2,0,2) = \frac{\pi^{4}}{45} + 
4 \pi^{2} \int_{0}^{1} A_{1}(q) B_{2}(q) K(q) \, dq + 
4 \pi^{2} \int_{0}^{1} A_{2}(q) B_{1}(q) K(q) \, dq,
\end{equation}
\end{example}
\noindent
and we have the evaluation
\begin{equation}
\int_{0}^{1} \left[A_{1}(q) B_{2}(q) + A_{2}(q) B_{1}(q)\right]
\log\sin\pi q \, dq
=\frac{\pi^2}{288}.
\end{equation}

\begin{example}
\label{T206-a}
Theorem \ref{Thm:T(m,0,n)} gives 
\begin{equation}
T(2,0,6) = \frac{\pi^{8}}{8100} + \frac{8 \pi^{7}}{45} I_{AB}(1,6) + 
\frac{8 \pi^{7}}{15} I_{AB}(2,5).
\end{equation}
\end{example}

\medskip

\begin{Thm}
Every Tornheim sum can be expressed as a finite sum of 
integrals of the form $I_{AB}$. 
\end{Thm}

\bigskip

\section{A representation for $I_{AB}(m,n)$ in terms of Clausen functions}
\label{S:N even - Clausen}

In this section we consider the integral 
\begin{equation}
I_{AB}(m,n)  =   \frac{1}{\pi}\int_{0}^{1} A_{m}(q) B_{n}(q) K(q) \, dq,
\label{IAB} 
\end{equation}
\noindent
and express it in terms of the family introduced in \eqref{def-X},
\begin{equation}
X_{k,l} := (-1)^{\lfloor{l/2 \rfloor}} 
\frac{l!}{(2\pi)^l}\ione  \log \Gamma(q) B_{k}(q) 
\Cl_{l+1} (2 \pi q)  \, dq. 
\label{def-X1-a}
\end{equation}
Here $\Cl_l(x)$ are the Clausen functions defined in
\eqref{def-clausen-even} and \eqref{def-clausen-odd}.

\medskip

The reduction of the integrals $I_{AB}$ employs 
a new family of special functions, $K_{n}(q)$, 
$n \in \mathbb{N}_{0}$, defined as the iterated primitives of the kernel 
\begin{equation}
K_{0}(q) := - K(q) = - \log \sin \pi q,
\end{equation}
\noindent
that is,
\begin{equation}
K_n(q)  :=  n \int_0^q K_{n-1}(q')dq',\quad n \geq 1. \label{Kn-def}
\end{equation}
\noindent
The functions $K_{n}(q)$ are 
normalized by the condition $K_n(0)=0$ for $n\ge 1$.

\begin{Note}
The plan is to use $(n+1)K_{n}(q) = K_{n+1}'(q)$ and integrate by parts
to transform the expression for $I_{AB}(m,n)$ into a finite sum of integrals
in which the only type-A function that appears is 
$A_{1}(q) = \log \Gamma(q) - \log \sqrt{2 \pi}$. \\

The Fourier expansion of $K_{0}(q)$ is
\begin{equation}
K_{0}(q) = \log 2 + \frac{1}{2} 
\Cl_{1} (2 \pi q),
\end{equation}
\noindent
where $\Cl_{1}$ is a Clausen function. 
We first show that $K_n(q)$ is the sum 
of a polynomial in $q$ of degree $n$ and a multiple of the Clausen
function $\Cl_{n+1}(2\pi q)$ (see \eqref{Kn-Fourier}).
We then show that the polynomial part can be integrated out explicitly in
terms of zeta values, leading finally to the analytic expression for the
Tornheim sums $T(m,0,n)$ of even weight $N = m+n$ in terms of the integrals
$X_{kl}$, given in Theorem \ref{really-final}.
\end{Note}

\medskip

We first recall some basic properties of the functions $A_m(q)$, introduced
in \cite{espmoll2}. The function 
\begin{equation}
A_{m}(q) := m \frac{\partial }{\partial z} \zeta(z,q) \Big{|}_{z=1-m}
\end{equation}
\noindent
satisfies, for $m\in\allN - \{1\}$, the recurrence
\begin{equation}
\frac{d}{dq} A_{m}(q) = mA_{m-1}(q) + \frac{m}{m-1} B_{m-1}(q),
\label{A-der}
\end{equation}
\noindent
with initial condition
\begin{equation}
A_{1}(q) = \log \Gamma(q) - \log \sqrt{2 \pi}. 
\label{A1-def}
\end{equation}
\noindent
The boundary values are given by 
\begin{equation}
A_{m}(0) = A_{m}(1) = m \zeta'(1-m), \text{ for } m \geq 2,
\label{boundary-A}
\end{equation}
\noindent
and 
\begin{equation}
A_{1}(0^{+}) = - \infty, \qquad A_{1}(1) = - \log \sqrt{2\pi}.
\label{boundary-A1}
\end{equation}
The recurrence (\ref{A-der}) is similar to the relation
\begin{equation}
\frac{d}{dq}B_m(q) = m B_{m-1}(q),
\label{B-der}
\end{equation}
\noindent
satisfied by the Bernoulli polynomials.

\begin{Thm}
The integral $I_{AB}(m,n)$ can be expressed in terms of 
the family
\begin{equation}
U_{k,l} := \int_{0}^{1} \log \Gamma(q) B_{k}(q) K_{l}(q) \, dq,
\label{U-def}
\end{equation}
with $k+l = m+n-1$.
\end{Thm}
\begin{proof}
The proof is by induction on $m$. We show first 
that any integral of the form
\begin{equation}
\int_{0}^{1} A_{m}(q) B_{n}(q) K_{p}(q) \, dq
\notag
\end{equation}
can be expressed in terms of 
\begin{equation}
U^{*}_{k,l} := \int_{0}^{1} A_{1}(q) B_{k}(q) K_{l}(q) \, dq,
\end{equation}
with $k+l+1=m+n+p$,
and then transform into $U_{k,l}$ by using $A_{1}(q) = \log \Gamma(q) 
- \log \sqrt{2 \pi}$. \\

Start with
\begin{equation}
\notag
(p+1)\int_{0}^{1} A_{m}(q) B_{n}(q) K_{p}(q) \, dq =
\int_{0}^{1} A_{m}(q) B_{n}(q) \frac{d}{dq} K_{p+1}(q) \, dq. 
\end{equation}
\noindent
Integrate by parts and use $K_{p+1}(0)=0$, (\ref{A-der}), \eqref{boundary-A}
and (\ref{B-der})  to derive
\begin{equation}\label{step-1}
\begin{split}
(p+1)\int_{0}^{1} A_{m}(q) B_{n}(q) K_{p}(q) \, dq &=
A_{m}(1)B_{n}(1)K_{p+1}(1) \\
&\quad - m \int_{0}^{1} A_{m-1}(q) B_{n}(q) K_{p+1}(q) \, dq   \\
&\quad - \frac{m}{m-1} \int_{0}^{1} B_{m-1}(q) B_{n}(q) K_{p+1}(q) \, dq\\
&\quad- n \int_{0}^{1} A_{m}(q) B_{n-1}(q) K_{p+1}(q).
\end{split}
\end{equation}

The term involving the product of two Bernoulli polynomials
can be evaluated in closed form, and only the last term in 
(\ref{step-1}) requires further analysis. Repeating the process to this 
last term produces
\begin{multline*}
(p+1)(p+2) \ione A_{m}(q)B_{n-1}(q) K_{p+1}(q) \, dq = A_{m}(1)B_{n-1}(1)K_{p+2}(1) 
\\
-  m \int_{0}^{1} A_{m-1}(q) B_{n-1}(q) K_{p+2}(q) \, dq 
-  \frac{m}{m-1}  \int_{0}^{1} B_{m-1}(q) B_{n-1}(q) K_{p+2}(q) \, dq \\
-  (n-1)  \int_{0}^{1} A_{m}(q) B_{n-2}(q) K_{p+2}(q) \, dq.  
\end{multline*}

We conclude that each integration by parts produces terms that can be 
evaluated in the stated closed form 
(by the induction hypothesis) and one integral
where the index of the Bernoulli polynomial term is decreased by one. 
Note that the indices $j,k,l$ of the three functions in any term of the 
integrand appearing in this procedure satisfy $j+k+l=m+n+p$.
The case of interest, $I_{AB}(m,n)$, corresponds to $p=0$. \\

Eventually we arrive at 
\begin{equation*}
\begin{split}
(p+n+1) \int_{0}^{1} A_{m}(q)&B_{0}(q) K_{p+n}(q) \, dq = 
A_{m}(1)B_{0}(1) K_{p+n+1}(1) \\
& -\int_{0}^{1} K_{p+n+1}(q) B_{0}(q) \left( m A_{m-1}(q) + \frac{m}{m-1} 
B_{m-1}(q) \right) \, dq  \\
& -\int_{0}^{1} A_{m}(q) \frac{d}{dq}B_{0}(q) K_{p+n+1}(q) \, dq .
\end{split}
\end{equation*}
\noindent
The fact that $B_{0}(q) \equiv 1$, so that the last integral vanishes 
identically, completes the argument. \\

It remains to prove that integrals of the form
\begin{equation}
V_{k,m,n} := \int_{0}^{1} B_{k}(q)B_{m}(q) K_{n}(q) \, dq 
\end{equation}
\noindent
can be evaluated in closed form and to use the relation (\ref{A1-def}) 
to eliminate $A_{1}(q)$ from the formulas. \\

The first step in this part of the proof
is to use the relation

\begin{multline}
B_{n_{1}}(q) B_{n_{2}}(q)  =  
\sum_{k=0}^{k(n_{1},n_{2})} \left[ n_{1} \binom{n_{2}}{2k} + n_{2} 
\binom{n_{1}}{2k} \right] 
\frac{B_{2k}}{n_{1}+n_{2}-2k} B_{n_{1}+n_{2}-2k}(q)   \\
 +  (-1)^{n_{1}+1} \frac{n_{1}! n_{2}!}{(n_{1}+n_{2})!} B_{n_{1}+n_{2} },
\end{multline}

\noindent
with 
\begin{equation}
k(n_{1},n_{2}) = \text{Max} \left\{ \lfloor n_{1}/2 \rfloor \}, 
 \lfloor n_{2}/2 \rfloor \} \right\},
\end{equation}
\noindent
for the product of two Bernoulli polynomials in order to reduce $V_{k,m,n}$
to sums of terms of the form 

\begin{equation}
W_{m,n} := \ione B_{m}(q) K_{n}(q) \, dq.
\label{W-def}
\end{equation}

\noindent
Integration by parts show that (\ref{W-def}) can be written in terms of 
the $W$-integrals with second index $0$. Indeed,
\begin{eqnarray}
W_{m,n} & = & \frac{1}{m+1} \ione K_{n}(q) \frac{d}{dq}B_{m+1}(q) \, dq 
\nonumber \\
 & = &\frac{B_{m+1}(1)K_{n}(1)}{m+1} - 
\frac{1}{m+1} \ione B_{m+1}(q) \frac{d}{dq} K_{n}(q) \nonumber \\
 & = &\frac{B_{m+1}(1)K_{n}(1)}{m+1} - 
\frac{n}{m+1} W_{m+1,n-1}, \nonumber 
\end{eqnarray}
\noindent
and iterating this procedure yields
\begin{equation}
W_{m,n} = \sum_{k=1}^{n} \frac{(-1)^{k+1}}{(m+1)_{k}} B_{m+k}(1)K_{n-k+1}(1)
+ \frac{(-1)^{n}}{(m+1)_{n}} W_{m+n,0},
\end{equation}
\noindent
where $(m)_{k} = m(m+1) \cdots (m+k-1)$ is the Pochhammer symbol.

The closed form of $V_{k,m,n}$ follows from
\begin{equation}
W_{r,0} = \int_{0}^{1} B_{r}(q) K_{0}(q) \, dq =
\begin{cases} 
\log 2  & r = 0, \\
0  & r \mbox{ odd}, \\
(-1)^{r/2} r!  \zeta(r+1)/(2 \pi)^{r}, & r > 0 \mbox{ even},
\end{cases}
\end{equation}
\noindent
given as Example $5.2$ in \cite{espmoll1}.
\end{proof}

\begin{Note}
The boundary terms $K_{n}(1)$ are given in (\ref{boundary-B}).
\end{Note}

Our final step in the reduction of the integrals $I_{AB}$
makes use of a representation of the kernels $K_{n}(q)$ in 
terms of the Clausen functions $\Cl_{n}(x)$, defined in 
(\ref{def-clausen-even}) and (\ref{def-clausen-odd}). 

\begin{Prop}
The kernels $K_n(q)$ are given by
\begin{equation}
K_0(q)= \log{2}+\frac{1}{2}\Cl_{1} (2\pi q),
\label{K0-Clausen}
\end{equation}
\noindent
and, for $n \geq 1$, 
\begin{equation}
K_n(q)=q^n\log{2}+{n!}\sum_{k=1}^{\lfloor n/2 \rfloor}
\frac{(-1)^{k+1}\zeta(2k+1)}{(2\pi)^{2k}(n-2k)!}q^{n-2k}
+p_n\frac{n!}{(2\pi)^n}\Cl_{n + 1} (2\pi q),
\label{Kn-Fourier}
\end{equation}
\noindent
where the coefficient $p_{n} = (-1)^{\lfloor{n/2 \rfloor}}$ is 
the {\em $4$-parity} of $n$ defined in (\ref{4-parity}). 
\end{Prop}

\begin{proof}
The Fourier series expansion of the 
function $K_0(q) = -\log\sin(\pi q)$ 
\begin{equation}
K_0(q) = \log{2}+\frac{1}{2}\sum_{k=1}^{\infty}\frac{\cos(2\pi k q)}{k}.
\label{K0-Fourier}
\end{equation}
\noindent
is standard. To
compute the expansion for $K_{1}(q)$, we cannot just integrate term by 
term the series (\ref{K0-Fourier}), since it is not 
uniformly convergent for $q \in 
[0,1]$. To bypass this problem, observe that the indefinite integral 
of $K_{0}(q)$ can be written as 
\begin{equation}
\int {K_0 (q)dq = q\log \left( {1 - e^{2\pi qi} } \right)}  
- q\log \sin \left( {\pi q} \right) - \frac{{\pi i}}
{2}q^2  - \frac{i}{{2\pi }}\operatorname{Li} _2 \left( {e^{2\pi qi} } \right),
\end{equation}
\noindent
where $\operatorname{Li} _n(z)$ is the polylogarithmic function defined
for $z\in\allC, |z|\le 1$ by
\begin{equation}
\operatorname{Li}_n(z) := \sum_{k=1}^{\infty}\frac{z^k}{k^n},\quad 
n\ge 2.
\end{equation}
Separating the real and imaginary parts we obtain, for $0\le q \le 1$,
\begin{multline}
\int K_0 (q)dq = 
q\log{2}+\frac{1}{2\pi}\sum_{k=1}^{\infty}\frac{\sin (2\pi k q)}{k^{2}}\cr
+i\left[\frac{\pi}{2}q(1-q)+
\frac{1}{2\pi}\sum_{k=1}^{\infty}\frac{\cos(2\pi k q)}{k^2}\right],
\end{multline}
and the proposition is proved for $n=1$ since the term in 
brackets is a $q$-independent constant. This can be seen directly from 
Hurwitz's Fourier series representation of the Bernoulli polynomials
\begin{equation}
B_{n}(q) = - \frac{n!}{(2 \pi i)^{n}} \sideset{}{'}\sum_{k= - \infty}^{\infty} 
\frac{e^{2 \pi i k q}}{k^{n}},
\end{equation}
\noindent
where the prime indicates that the term $k = 0$ must be 
excluded in the sum.

The result for $n\ge 2$ can be obtained directly from the expression 
for $K_1(q)$,
\begin{equation}
K_1(q)=q\log{2}+\frac{1}{2\pi}\sum_{k=1}^{\infty}\frac{\sin (2\pi k q)}{k^{2}},
\label{K1-Fourier}
\end{equation}
\noindent
integrating successively term by term, which is now legitimate
since, for $n \geq 2$,  the series defining the Clausen function 
$\Cl_{n}(x)$ is uniformly convergent.
\end{proof}

\begin{Note}
The evaluation of the basic integral $I_{AB}(m,n)$ requires the boundary 
values of $K_n(q)$, $n\in\allN$, at $q=1$. These are given by 
\begin{equation}
K_n(1)=\log{2}+{n!}\sum_{k=1}^{\lfloor (n-1)/2 \rfloor}
\frac{(-1)^{k+1}\zeta(2k+1)}{(2\pi)^{2k}(n-2k)!},\qquad n\in\allN,
\label{boundary-B}
\end{equation}
where we have used the special values 
$\Cl_{2n}(2\pi)=0$ and 
$\Cl_{2n+1}(2\pi)=\zeta(2n+1)$. \\
\end{Note}

\begin{Note}
The method presented in this section reduces the evaluation of all 
{\em even}-weight Tornheim sums to the evaluation of the integrals $U_{m,n}$,
defined in \eqref{U-def}.
The expansion (\ref{Kn-Fourier}) and the formula
\begin{multline}
\ione q^{n} \log \Gamma(q) \, dq = 
\frac{1}{n+1} \sum_{k=1}^{\lfloor{ \tfrac{n+1}{2} \rfloor} } (-1)^{k} 
\binom{n+1}{2k-1} \frac{(2k)!}{k (2 \pi)^{2k}} \left[ \delta \zeta(2k) - 
\zeta'(2k) \right] \\
- \frac{1}{n+1} \sum_{k=1}^{\lfloor{ \tfrac{n}{2} \rfloor}} 
(-1)^{k} \binom{n+1}{2k} \frac{(2k)!}{2(2 \pi)^{2k}} \zeta(2k+1) + 
\frac{\log \sqrt{2 \pi}}{n+1}, 
\end{multline}
\noindent
(with $\delta = 2 \log \sqrt{2 \pi} + \gamma$, where $\gamma$ is Euler 
constant) given as (6.14) in \cite{espmoll1}, reduces the 
evaluation of Tornheim sums
to the evaluation of the integrals $X_{m,n}$, introduced in \eqref{def-X1-a}.
The result described in 
Theorem \ref{really-final} appears from transforming the explicit representation
\eqref{T(m,0,n)-explicit-1} in Theorem \ref{Thm:T(m,0,n)}
by the methods presented in this section. The details are left to the 
reader. 
\end{Note}

\begin{Note}\label{Note-X0n}
The value 
\begin{equation}
X_{0,n} = (-1)^{ \lfloor{ n/2 \rfloor}} \frac{n!}{(2 \pi)^n} 
\ione \log \Gamma(q) \Cl_{n+1}(2 \pi q) \, dq
\end{equation}
\noindent
can be obtained from the integrals 
\begin{align*}
\ione \log \Gamma(q) \sin( 2 \pi k q ) \, dq & =  \frac{A + \log k}{2 \pi k},\cr
\ione \log \Gamma(q) \cos( 2 \pi k q ) \, dq & =  \frac{1}{4 k},
\end{align*}
\noindent
where $A = \log(2 \pi) + \gamma$, with $\gamma$ the Euler constant. This 
appears in \cite{gr} 6.443.1 and 6.443.3. It follows that 
\begin{equation}
\label{X0-value}
X_{0,n} = (-1)^{ \lfloor{ n/2 \rfloor}} \frac{n!}{(2 \pi)^n} 
\times 
\begin{cases}
\frac{1}{2 \pi} \left( A \zeta(n+2) - \zeta'(n+1) \right), \quad 
\text{ for } n \text{ odd}, \\
\frac{1}{4} \zeta(n+2), \quad \text{ for } n \text{ even}. 
\end{cases}
\end{equation}
For our evaluations of even-weight Tornheim sums, we will only need the result for $n$ even. 
\end{Note}

\begin{Example}
\label{T202-b}
This is a continuation of Example \ref{T202-a}.  We express here 
the sum $T(2,0,2)$ first in terms of the $U$-integrals. 
Theorem \ref{Thm:T(m,0,n)}, with $m = n = 2$, gives
\begin{multline*}
T(2,0,2) = \zeta (2)^2  - \frac{1}
{2}\zeta (4) - 4\pi ^2 \left[ \int_0^1 {A_1 (q)B_2 (q)K_0 } (q)dq\right.\\ 
+ \left. \int_0^1 {A_2 (q)B_1 (q)K_0 } (q)dq\right].
\end{multline*}

The second integral is now modified to avoid the presence of the function 
$A_{2}(q)$. Write 
$K_0(q)=(d/dq)K_1(q)$, integrate by parts and use
formulas (\ref{A-der}) and (\ref{B-der}) to  get 
\begin{multline*}
\int_{0}^{1}  A_{2}(q)B_{1}(q)K_{0}(q)dq = A_{2}(1)B_{1}(1)K_{1}(1)
-\int_{0}^{1} A_{2}(q)B_{0}(q)K_{1}(q)dq\\
-2\int_{0}^{1} A_{1}(q)B_{1}(q)K_{1}(q)dq
-2\int_{0}^{1} B_{1}^2(q)K_{1}(q)dq.
\end{multline*}
The new integral involving $A_2(q)B_{0}(q)K_{1}(q)$ can be 
dealt with in a similar manner, to produce
\begin{equation*}
\begin{split}
\int_0^1 A_2 (q)B_1 (q)K_0(q)&\,dq = A_2(1)B_1(1)K_1(1)- \frac{1}{2} A_2(1)B_0(1)K_2(1)\cr
&+ \int_0^1 A_1 (q)B_0(q)K_2(q)\,dq -2\int_0^1 A_1 (q)B_1 (q)K_1(q)\,dq\cr
&+ \int_0^1 B_0(q)B_1 (q)K_2(q)\,dq -2\int_0^1 B_1^2 (q)K_1(q)\,dq.
\end{split}
\end{equation*}
\noindent
From (\ref{boundary-A}) and (\ref{boundary-B}) we see that the boundary
terms calcel out and therefore the Tornheim sum $T(2,0,2)$ is given by
\begin{equation}
T(2,0,2) = \zeta^{2}(2) - \zeta(4)/2
- 4 \pi^{2} \left( U^{*}_{2,0} - 2U^{*}_{1,1} + U^{*}_{0,2} + V_{0,1,2} - 2V_{1,1,1} 
\right). \nonumber
\end{equation}
\noindent
The expression for $V_{k,m,n}$ shows that
$V_{0,1,2} = 2 V_{1,1,1} = \frac{1}{12} \log 2$. Therefore, the sum 
$T(2,0,2)$ is expressed in terms of the $U^{*}$-integrals as
\begin{equation*}
T(2,0,2) = \frac{\pi^{4}}{45}
- 4 \pi^{2} \left( U^{*}_{2,0} - 2U^{*}_{1,1} + U^{*}_{0,2} \right).
\end{equation*}
\noindent
The transformation from $U^{*}_{m,n}$ to $U_{m,n}$ gives
\begin{equation*}
T(2,0,2) = \frac{\pi^{4}}{45}
+\frac{1}{3} \log (2 \pi ) \left(\pi ^2 \log (2)+9 \zeta (3)\right)
- 4 \pi^{2} \left( U_{2,0} - 2U_{1,1} + U_{0,2} \right).
\end{equation*}
The final transformation is to write $U_{m,n}$ in terms 
of $X_{m,n}$ to obtain
\begin{equation}
T(2,0,2) = \frac{\pi^{4}}{45} - Y_{2,2}^{*},
\end{equation}
\noindent
where
\begin{equation}
Y_{2,2}^{*} := 4 \pi^{2} \left( X_{0,2} - 2X_{1,1} + X_{2,0} \right) 
-2 \zeta(3) \log(2 \pi).
\end{equation}
\end{Example}

\begin{Example}
At weight 8, we consider first the value of 
\begin{equation*}
T(2,0,6) = \frac{\pi^{8}}{8100} + \frac{4 \pi^{7}}{45}
\left( 2 I_{AB}(1,6) + 6 I_{AB}(2,5) \right).
\end{equation*}
\\
Expressing the $I_{AB}$ integrals in terms of the $U^{*}$ integrals we find
\begin{multline*}
2 I_{AB}(1,6) + 6 I_{AB}(2,5) = \\
-\frac{2}{\pi}\left( U_{0,6}^* - 6 U_{1,5}^* 
+ 15 U_{2,4}^* - 20 U_{3,3}^* + 15 U_{4,2}^* - 6U_{5,1}^* + U_{6,0}^* \right),
\end{multline*}
and in terms of the $U$ integrals:
\begin{multline*}
2 I_{AB}(1,6) + 6 I_{AB}(2,5) =
\frac{1}{2\pi} \log (2 \pi ) \left(\frac{\log (2)}{21}-\frac{\zeta (3)}{2 \pi ^2}-\frac{15 \zeta (5)}{2
   \pi ^4}+\frac{315 \zeta (7)}{2 \pi ^6}\right)\\
   -\frac{2}{\pi}\left( U_{0,6} - 6 U_{1,5} + 15 U_{2,4} - 20 U_{3,3} + 15 U_{4,2} - 6U_{5,1} + U_{6,0} \right).
\end{multline*}

\noindent
Thus we obtain the representation
\begin{multline}
T(2,0,6) = \frac{\pi^{8}}{8100} + \frac{2 \pi^{6}}{45}
\log (2 \pi ) \left(\frac{\log (2)}{21}-\frac{\zeta (3)}{2 \pi ^2}-\frac{15 \zeta (5)}{2
   \pi ^4}+\frac{315 \zeta (7)}{2 \pi ^6}\right)\\
   - \frac{8 \pi^{6}}{45}\left( U_{0,6} - 6 U_{1,5} + 15 U_{2,4} - 20 U_{3,3} + 15 U_{4,2} - 6U_{5,1} + U_{6,0} \right).
\end{multline}

\noindent
In terms of the $X_{m,n}$ integrals we have
\begin{equation}
T(2,0,6) = \frac{\pi^{8}}{8100} - Y_{6,2}^{*},
\end{equation}
where
\begin{multline}
Y_{6,2}^{*} := \frac{8 \pi^{6}}{45}\left( X_{0,6} - 6 X_{1,5} + 15 X_{2,4} - 
20 X_{3,3} + 15 X_{4,2} - 6X_{5,1} + X_{6,0} \right)\cr
 - 6 \zeta(7) \log(2 \pi).
\end{multline}

\end{Example}

\noindent
\begin{Example}
The other two Tornheim sums of weight $8$ are $T(3,0,5)$ and $T(4,0,4)$. They 
are given by 
\begin{align*}
T(3,0,5) &= -\frac{\pi^{8}}{18900}   + \zeta(3) \zeta(5)
+ \frac{8 \pi^{7}}{45} \left( 5 I_{AB}(4,3) + 3  I_{AB}(5,2) \right)\cr
\intertext{and}
T(4,0,4) &= \frac{\pi^{8}}{14175} + \frac{8 \pi^{7}}{9}
\left( I_{AB}(3,4) + I_{AB}(4,3) \right),
\end{align*}
\noindent
respectively.
\end{Example}
\noindent
Expressing the $I_{AB}$ integrals in terms of the $X$ integrals we find
\begin{align}
T(3,0,5) &= -\frac{\pi^{8}}{18900}  - 6\zeta(3) \zeta(5)
- Y_{3,5}^{*},\\
\intertext{where}
Y_{3,5}^{*} &:= \frac{32 \pi^{6}}{9} 
\left( X_{0,6} - 3 X_{1,5} + 3 X_{2,4} - X_{3,3} \right) 
- 15 \zeta(7) \log(2 \pi),\\
\intertext{and}
T(4,0,4) &= \frac{\pi^{8}}{14175} - 4\zeta(3) \zeta(5)
- Y_{4,4}^{*},\\
\intertext{where}
Y_{4,4}^{*} &:= \frac{8 \pi^{6}}{3} 
\left( X_{0,6} - 4 X_{1,5} + 6 X_{2,4} - 4 X_{3,3} + X_{4,2} \right)
- 20 \zeta(7) \log(2 \pi).
\end{align}

\medskip
\begin{Note}
In the examples given above, the integrals $X_{k,l}$
appear in a symmetric form. This is a general feature, as stated in
Theorem \ref{really-final}. Depending on the parity of the integers $m$
and $n$, the even-weight Tornheim sum $T(m,0,n)$ contains either $Y_{m,n}$
or $Y_{n,m}$, where
\begin{equation*}
Y_{m,n} := \frac{2(2\pi)^{m+n-2}}{m!(n-2)!} \sum_{j=0}^{m} (-1)^{j} \binom{m}{j} X_{j,m+n-2-j}.
\label{def-Y-1}
\end{equation*}
There are linear relations among the integral $Y_{m,n}$ of fixed
weight $N:=m+n$. These come from the linear relations among the Tornheim
sums $T(m,0,n)$ of the same weight, discussed in the next section.
For example, the identity
\begin{equation*}
T(2,0,6) +  T(6,0,2) =  \frac{2}{3} \zeta(8)
\end{equation*}
\noindent
gives
\begin{equation*}
Y_{2,6}+Y_{6,2} = 12 \log (2
   \pi ) \zeta (7) + \frac{7}{3} \zeta (8)  - 6 \zeta (3) \zeta (5),
\end{equation*}
whereas
\begin{equation*}
5T(6,0,2) + 2 T(5,0,3) =  \frac{163}{12} \zeta(8) -8 \zeta(3) \zeta(5) 
\end{equation*}
\noindent
gives
\begin{equation*}
5 Y_{2,6} + 2 Y_{5,3} = 60 \log (2 \pi ) \zeta
   (7)+\frac{29}{6}\zeta(8)- 30\zeta (3) \zeta (5).
\end{equation*}
\end{Note}
A systematic study of the integrals $Y_{m,n}$ will be presented elsewhere. 

\bigskip

\section{A systematic list of examples} \label{S:examples}

Here we present a systematic evaluation of the Tornheim sums 
$T(m,k,n)$ with $m, \, k, \, n \in \mathbb{N} \cup \{0 \}$. The sums are 
organized according to the weight $N = m+k+n$.  The conditions 
$m+n \geq  2, \, k+n \geq 2$ and $N \geq 3$ are imposed for 
convergence. The symmetry  relation 
$T(m,k,n) = T(k,m,n)$ is used to impose $m \geq k$. 

We use the notation $\mathcal{Z}_{N}$ and $\mathcal{Z}_{N}^{0}$
introduced in the introduction.
Huard's result \eqref{huard-1} allows us to evaluate any Tornheim
sum in $\mathcal{Z}_{N}$ in terms of the sums in $\mathcal{Z}_{N}^{0}$.
Before detailing a systematic algorithm to evaluate all the sums in
$\mathcal{Z}_{N}$ for a given weight $N$ and giving specific examples, 
we shall determine how many of the Tornheim sums in $\mathcal{Z}_{N}^{0}$, for
a given {\em even} weight $N$, remain
undetermined after using all the known algebraic identities for these sums.

As we discussed in the introduction, two of the sums in $\mathcal{Z}_{N}^{0}$
have known explicit evaluations:
\begin{equation}
T(0,0,N) = \zeta(N-1)-\zeta(N),\quad N\ge 3,
\label{torn-two-1}
\end{equation}
and
\begin{equation}
T(1,0,N-1) = \frac{1}{2} \left[ (N-1) \zeta(N) - 
\sum_{i=2}^{N-2} \zeta(i) \zeta(N-i) \right], \quad N\ge 3.\label{tor-new3-1}
\end{equation}
\noindent

\noindent
Euler proved that for $m\ge 2, n\ge 2$ the sums $T(m,0,n)$ satisfy the 
symmetrized identity
\begin{equation}
T(m,0,n) + T(n,0,m) = \zeta(m)\zeta(n) - \zeta(m+n),
\label{euler1}
\end{equation}
so that only the case $m\ge n$ needs to be considered.
\noindent
In particular, for $m=n$ we find
\begin{equation}
T(n,0,n) = \frac{1}{2} \zeta^{2}(n) - \frac{1}{2}\zeta(2n).
\label{euler2-0}
\end{equation}

At even weight $N$, the previous identities leave exactly $N/2-2$ Tornheim
sums of the type $T(m,0,n)$ unevaluated. In effect, the convergence
conditions require $n\ge 2$, so that all the sums with $m = N/2+1,
N/2+2,\ldots,N-2$ are undetermined.
\noindent
The remaining $N/2-2$ unevaluated Tornheim sums are not all independent,
since they satisfy linear relations obtained by applying
Huard's identity \eqref{huard-1} to some special cases with
known evaluations as, for instance,
\begin{align}
T(m,k,0) &= \zeta(m)\zeta(k), \\
T(m,k,1) &= (-1)^m \left\{ \sum_{i=2}^{m} (-1)^i \zeta(i) \zeta(N-i) + 
\right.\cr
&\qquad\qquad\qquad\left. \frac{1}{2} \sum_{i=2}^{N-2} \zeta(i) \zeta(N-i) -
\frac{1}{2}(N+1) \zeta(N) \right \}, \\
\intertext{or}
T(m,1,1) &= \frac{1}{2} \left( (N+1) \zeta(N) 
- \sum_{i=2}^{N-2} \zeta(i) \zeta(N-i) \right).
\end{align}
These results appear in \cite{tornheim1}. As usual, $N$ is 
the weight of the Tornheim sum in the left hand side. 

\medskip

On the other hand, Granville
\cite{granville1} showed that the multiple zeta values 
defined in (\ref{def-MZV}) satisfy 
\begin{equation}
\sum \zeta(p_{1}, p_{2}, \cdots, p_{g} ) = \zeta(N),
\end{equation}
\noindent
where the sum is over all elements of $\mathbb{Z}^{g}$ such that 
$p_{1} + p_{2} + \cdots + p_{g} = N$, with $p_{j} \geq 1$ and $p_{1} \geq 2$. 
In particular, when $g=2$ we obtain
\begin{equation}
\sum_{m+n = N} T(m,0,n) = \zeta(N),
\end{equation}
\noindent
where the sum is over pairs $(m,n)$ with $m \geq 1$ and $n \geq 2$.

\medskip

\noindent
{\bf Experimental observation}. Not all of the relations 
stated above are linearly independent.
However, it is possible to show that all Tornheim sums
in $\mathcal{Z}_{N}^{0}$ for weights $N=4$ and $N=6$ can be completely
evaluated in terms of zeta values (the details are provided at the 
end of this section). In the case of
even weight $N\ge 8$ the Tornheim sums 
can be expressed in terms of zeta
values and a reduced number of basis sums, which we choose as 
$T(N-2k,0,2k)$, with $k=1,\ldots,K$,
where $K$ is given by
\begin{equation}
K = \left\lfloor \frac{N-2}{6} \right\rfloor
\end{equation}
In particular, we need only one basis sum for weights 8, 10 and 12; two 
for weights 14, 16 and 18, and  three for weights 20, 22 and 24. The reader 
will find in \cite{borwein95} that $K$ given above is an upper bound for 
the number of Tornheim sums required. 

\medskip

\noindent
\begin{example}
Indeed, for weight $N=8$ we find that all sums in $\mathcal{Z}_{8}^{0}$
have either explicit evaluations or can be expressed 
in terms of just $T(6,0,2)$:
\begin{align*}
 T(0,0,8)&= \zeta (7)-\zeta (8), \\
 T(1,0,7)&= \tfrac{5}{4}\zeta (8)-\zeta (3) \zeta(5), \\
 T(2,0,6)&= \tfrac{2}{3}\zeta (8)-T(6,0,2),\\
 T(3,0,5)&= -\tfrac{187}{24}\zeta (8)+5 \zeta (3) \zeta(5)+\tfrac{5}{2} T(6,0,2),\\
 T(4,0,4)&= \tfrac{1}{12}\zeta (8) ,\\
 T(5,0,3)&= \tfrac{163}{24}\zeta (8)-4 \zeta (3) \zeta(5)-\tfrac{5}{2} T(6,0,2).
\end{align*}
Note that the identity \eqref{huard-7} follows directly from these expressions.\end{example}

\medskip

\begin{example}
For weight $N=14$ we can express all Tornheim sums as zeta values plus the 
sums $T(12,0,2)$ and $T(10,0,4)$: 
{\allowdisplaybreaks
\begin{equation*}
\begin{split}
 T(0,0,14)&= \zeta (13)-\zeta (14), \\
 T(1,0,13)&= \tfrac{11}{4} \zeta (14)-\zeta (3) \zeta (11)-\zeta (5) \zeta
   (9)-\tfrac{1}{2}\zeta (7)^2, \\
 T(2,0,12)&= \tfrac{271}{420} \zeta (14)-T(12,0,2), \\
 T(3,0,11)&= -\tfrac{35741}{840} \zeta (14)+11 \zeta
   (3) \zeta (11)+16 \zeta (5) \zeta (9)+9 \zeta (7)^2+\tfrac{11}{2} T(12,0,2),\\
 T(4,0,10)&= \tfrac{1}{12}\zeta (14)-T(10,0,4),\\ 
 T(5,0,9)&= \tfrac{40977}{112} \zeta (14)-\tfrac{165}{2} \zeta (3) \zeta (11)-147 \zeta (5) \zeta(9)-\tfrac{345}{4} \zeta (7)^2 \\
 &\qquad\qquad\qquad\qquad\qquad\qquad\qquad\qquad
  +\tfrac{9}{2} T(10,0,4)-\tfrac{165}{4} T(12,0,2),\\
 T(6,0,8)&= -\tfrac{20773}{35} \zeta (14)+132 \zeta(3) \zeta (11)+240 \zeta (5) \zeta (9)+141 \zeta (7)^2 \\
 &\qquad\qquad\qquad\qquad\qquad\qquad\qquad\qquad
  -6 T(10,0,4)+66 T(12,0,2),\\
 T(7,0,7)&= \tfrac{1}{2}\zeta (7)^2-\tfrac{1}{2}\zeta (14), \\
 T(8,0,6)&= \tfrac{16619}{28} \zeta (14) -132 \zeta (3)\zeta (11)-240 \zeta (5) \zeta (9)-141 \zeta (7)^2\\
 &\qquad\qquad\qquad\qquad\qquad\qquad\qquad\qquad
  +6 T(10,0,4)-66 T(12,0,2),\\
 T(9,0,5)&= -\tfrac{41089}{112} \zeta (14) +\tfrac{165}{2} \zeta (3) \zeta (11)+148\zeta (5) \zeta (9)+\tfrac{345}{4} \zeta(7)^2\\
 &\qquad\qquad\qquad\qquad\qquad\qquad\qquad\qquad
  -\tfrac{9}{2} T(10,0,4)+\tfrac{165}{4} T(12,0,2), \\
 T(11,0,3)&= \tfrac{34901}{840} \zeta (14)-10 \zeta(3) \zeta (11)-16 \zeta (5) \zeta (9)-9 \zeta (7)^2\\
 &\qquad\qquad\qquad\qquad\qquad\qquad\qquad\qquad
  -\tfrac{11}{2} T(12,0,2).
\end{split}
\end{equation*}
}
\end{example}

\medskip
We are now in a position to formulate a systematic and exhaustive
algorithm to evaluate all the Tornheim sums in $\mathcal{Z}_{N}$.
The reader is invited to download the Tornheim Mathematica 6.0 package
developed by the authors, available at
\verb+http://www.math.tulane.edu/~vhm/packages.html+.
Most of the calculations in this paper can be easily reproduced with the 
aid of this package.

\medskip
\noindent
{\bf Algorithm}. The 
process of determining the Tornheim sums $T(m,k,n)$ proceeds as
follows.   \\

\noindent
{\bf Step 1}. First catalogue all sums that have known
explicit evaluations. All of these cases appear in 
\cite{tornheim1}. In this category we find:
\begin{align}
\intertext{The sums with third entry $n = 0$:}
T(m,k,0) &= \sum_{r=1}^{\infty} \sum_{s=1}^{\infty} 
\frac{1}{r^{m} \, s^{k}} = \zeta(m) \zeta(k).
\label{T(m,k,0)}
\intertext{The sums with third entry $n=1$:}
\begin{split}
\label{T(m,k,1)}
T(m,k,1) &= (-1)^m \left\{ \sum_{i=2}^{m} (-1)^i \zeta(i) \zeta(N-i) + 
\right.\cr
&\left. \frac{1}{2} \sum_{i=2}^{N-2} \zeta(i) \zeta(N-i) -
\frac{1}{2}(N+1) \zeta(N) \right \},
\end{split}
\intertext{where $N = m+k+1$ is the weight.}
\intertext{The sums of the type}
T(1,1,n) & = (n+1) \zeta(n+2) - \sum_{i=2}^{n} 
\zeta(i) \zeta(n+2-i).  \label{tor-new1}
\intertext{The following sums in $\mathcal{Z}_{N}^{0}$:}
T(0,0,n) &= \zeta(n-1)-\zeta(n),
\label{torn-two-0-a}\\
T(1,0,n) &= \frac{1}{2} \left( n \zeta(n+1) - 
\sum_{i=2}^{n-1} \zeta(i) \zeta(n+1-i) \right), \label{tor-new3-2}
\intertext{and the symmetric sum,}
T(m,0,m) &= \frac{1}{2} \zeta^{2}(m) - \frac{1}{2}\zeta(2m), \qquad m\ge 2.
\label{euler2}
\end{align}
For weight $N=m+n$ odd, the sum $T(m,0,n)$ is given by
\begin{multline}\label{huard-2a}
T(m,0,n) =  (-1)^{m} \sum_{j=0}^{\lfloor{ \frac{n-1}{2} \rfloor} } 
\binom{N-2j-1}{m-1} \zeta(2j) \zeta(N-2j)  \\
 +  (-1)^{m} \sum_{j=0}^{\lfloor{ \frac{m}{2} \rfloor} } 
\binom{N-2j-1}{n-1} \zeta(2j) \zeta(N-2j) -\tfrac{1}{2}\zeta(N).
\end{multline}

\medskip

\noindent
{\bf Step 2}. If $k \neq 0$, use the reduction of Huard et al. given in 
(\ref{huard-1}) to reduce the $T(m,k,n)$ not covered by the previous step
to a finite sum of Tornheim sums in $\mathcal{Z}_{N}^{0}$: \\
\begin{equation}
T(m,k,n) = \sum_{i=1}^{m} \binom{m+k-i-1}{m-i} T(i,0,N-i) + 
\sum_{i=1}^{k} \binom{m+k-i-1}{k-i} T(i,0,N-i),
\end{equation}
\noindent
with $N = m+k+n$.

\medskip

\noindent
{\bf Step 3}. For $N=m+n \le 6$ even and $m, \, n \geq 2$, compute all the sums
$T(m,0,n)$ explicitly by solving the set of simultaneous equations 
obtained from: 
(a) applying Huard's reduction of the previous step to the identity \eqref{T(m,k,0)},
for all independent pairs $(m,k)$ with $m+k=N$; (b) Euler's identity \eqref{euler1};
and using the known explicit evaluations of the sums in $\mathcal{Z}_{N}^{0}$
given in Step 1.

\medskip

\noindent
{\bf Step 4}. For $N=m+n \ge 8$ even and $m, \, n \geq 2$, we write all the sums
$T(m,0,n)$ in terms of the irreducible basis for weight $N$,
\begin{equation}
\left\{T(N-2k,0,2k),\quad k=1,\ldots,\left\lfloor \frac{N-2}{6} \right\rfloor \right\}.
\end{equation}

\medskip

\noindent
{\bf Step 5}. The irreducible sums of the previous step are evaluated in 
terms of
the integrals $X_{k,l}$ using Theorem \ref{really-final}. 

\medskip

For a given {\em even} weight $N=m+k+n$,
the whole process gives $T(m,k,n)$ as a finite sum (with rational
coefficients) of the zeta values $\zeta(N)$ and $\zeta(N-1)$, products 
of two zeta values of the form $\zeta(j) \zeta(N-j)$ with $2 \leq j \leq N-2$
and, for $N\ge 8$, a finite number of integrals the type $Y_{2r,N-2r}^*$, where
\begin{equation}
Y_{2r,N-2r}^* := Y_{2r,N-2r}+(-1)^{\frac{N}{2}-1}\binom{N-2}{2r-1}\zeta(N-1)\log2\pi.
\end{equation}

\begin{Definition}
We say that $(m,k,n)$ is an {\em admissible triple} if $m \geq k$ and
$m,k,n$ satisfy the conditions $m+n \geq  2, \, k+n \geq 2$ and $N \geq 3$ for the convergence of $T(m,k,n)$.
\end{Definition}

We present now the results for
small weight $N = m+k+n$.  The cases of weight $3$ and $4$ are
straightforward, as all admissible triples corresponds to cases
where the Tornheim sum has an explicit formula. \\

\noindent
\begin{center}
{\bf Weight $3$ }
\end{center}

\medskip\noindent
The admissible triples are $(0,0,3), \, (1,0,2)$ and 
$(1,1,1)$. We obtain
\begin{align*}
T(0,0,3) &= \zeta(2) - \zeta(3),\cr
T(1,0,2) &= \zeta(3),\cr
T(1,1,1) &= 2 \zeta(3),
\end{align*}
directly from \eqref{torn-two-1}, \eqref{tor-new3-1} and \eqref{tor-new1}, respectively.
\noindent
The reader will find in \cite{borwein-bradley1} many proofs of
the last identity.  \\

\begin{center}
{\bf Weight $4$ }
\end{center}

\medskip\noindent
The six admissible triples are
\begin{equation}
(0,0,4),  \, (1,0,3), \, (1,1,2), \,  (2,0,2), \, 
(2,1,1) \text{ and }  (2,2,0). \nonumber
\end{equation}

\medskip

\noindent We obtain
\begin{align*}
T(0,0,4) &= \zeta(3) - \zeta(4) = \zeta(3) - \frac{\pi^4}{90},\cr
T(1,0,3) &= \frac{1}{2} \left(3 \zeta (4)-\zeta (2)^2\right)
= \frac{1}{4}\zeta(4) = \frac{\pi^{4}}{360},\cr
T(1,1,2) &= 3 \zeta (4)-\zeta (2)^2 = \frac{1}{2}\zeta(4) = \frac{\pi^{4}}{180},\cr
T(2,0,2) &= \frac{1}{2} \left(\zeta (2)^2-\zeta (4)\right)
= \frac{3}{4}\zeta(4) = \frac{\pi^{4}}{120},\cr
T(2,1,1) &= \frac{1}{2} \left(5 \zeta (4)-\zeta (2)^2\right)
= \frac{5}{4}\zeta(4) = \frac{\pi^{4}}{72},\cr
T(2,2,0) &= \zeta (2)^2 = \frac{5}{2}\zeta(4) = \frac{\pi^{4}}{36},
\end{align*}
directly from \eqref{torn-two-1}, \eqref{tor-new3-1}, \eqref{tor-new1}, \eqref{euler2},
\eqref{T(m,k,1)} and \eqref{T(m,k,0)}, respectively.

\medskip
\noindent
The methods developed here produce the result
\begin{equation}
T(2,0,2) = \frac{\pi^{4}}{45} + 2 \log(2 \pi) \zeta(3) - Y_{2,2}. 
\end{equation}
\noindent
It follows that 
\begin{equation}
Y_{2,2}  = \frac{\pi^{4}}{72} + 
2 \log(2 \pi) \zeta(3).
\end{equation}
\noindent
As a consequence of this, we obtain the definite integral 
\begin{multline*}
\ione \left( 2 \pi^{2} B_{2}(q) \Cl_{1}(2 \pi q) 
- 2 \pi B_{1}(q) \Cl_{2}(2 \pi q) 
-  B_{0}(q) \Cl_{3}(2 \pi q)  \right) 
\log \Gamma(q) \, dq   \\ = 
\frac{\pi^{4}}{144} + 
\log(2 \pi) \zeta(3),
\end{multline*}
\noindent
where
\begin{equation*}
B_{0}(q) = 1, \;\; B_{1}(q) = q - 1/2, \;\; B_{2}(q) = q^{2}-q + 1/6,
\end{equation*}
and 
\begin{equation*}
\Cl_{1}(2 \pi q ) = \sum_{k=1}^{\infty} \frac{\cos 2 \pi k q }{k}, \; \;
\Cl_{2}(2 \pi q ) = \sum_{k=1}^{\infty} \frac{\sin 2 \pi k q }{k^{2}}, \;\;
\Cl_{3}(2 \pi q ) = \sum_{k=1}^{\infty} \frac{\cos 2 \pi k q }{k^{3}}.
\end{equation*}
Considering that from \eqref{X0-value} we also know that
\begin{equation}
\ione \log\Gamma(q)\Cl_{3}(2 \pi q )\,dq=\frac{1}{4}\zeta(4)=\frac{\pi^4}{360},
\end{equation}
we also find
\begin{multline}
\ione \left( 2 \pi^{2} B_{2}(q) \Cl_{1}(2 \pi q) 
- 2 \pi B_{1}(q) \Cl_{2}(2 \pi q) 
\right) 
\log \Gamma(q) \, dq   \\ = 
\frac{7\pi^{4}}{720} + 
\log(2 \pi) \zeta(3).
\end{multline}
\noindent
{\em Not your average integral}.  \\

\medskip

We state next the values of the Tornheim sums $T(m,k,n)$ of
even weight $N\ge 6$. The 
formulas are the direct output of the algorithms presented here, the only 
reductions used are those given at the beginning of this section. The reader
will observe that the final 
expresions contain a single even value of the Riemann zeta function. This is 
artificial. For instance, the value 
\begin{equation*}
T(4,2,0) = \tfrac{7}{4} \zeta(6),
\end{equation*}
\noindent
should be written as 
\begin{equation*}
T(4,2,0) = \zeta(4) \zeta(2),
\end{equation*}
\noindent
as in (\ref{T(m,k,0)}). This latter representation would be more helpful in 
the search for a closed-form expression for the Tornheim sums. However, at 
this point, we have decided to minimize the number of zeta values appearing in
the formulas.

\begin{Note}
In the examples that follow, we do not list the Tornheim sum $T(0,0,N)$. This 
is the only sum that explicitly contains the term $\zeta(N-1)$. 
\end{Note}

\bigskip

\begin{center}
{\bf Weight $6$ }
\end{center}

\medskip

In this example we give complete details, which will 
be omitted for higher weights.
The admissible triples are now
$(0,0,6)$, $(1,0,5)$, $(1,1,4)$, $(2,0,4)$, $(2,1,3)$, $(2,2,2)$, $(3,0,3)$, 
$(3,1,2)$, $(3,2,1)$, $(3,3,0)$, $(4,0,2)$, $(4,1,1)$ and $(4,2,0)$. \\

A direct application of the identities
and explicit formulas already discussed give the following evaluations
(in the formulas below we have
replaced the product $\zeta(2)\zeta(4)$ by $\frac{7}{4}\zeta(6)$):
 
\begin{align*}
T(1,0,5) & = \tfrac{3}{4} \zeta(6) - \tfrac{1}{2} \zeta^{2}(3), \quad &
T(1,1,4) & = \tfrac{3}{2} \zeta(6) - \zeta^{2}(3),  \\
T(3,0,3) & = - \tfrac{1}{2} \zeta(6) + \tfrac{1}{2} \zeta^{2}(3), \quad &
T(3,3,0) & = \zeta^{2}(3), \\
T(4,1,1) & = \tfrac{7}{4}\zeta(6) - \tfrac{1}{2} \zeta^{2}(3), \quad &
T(4,2,0) & = \tfrac{7}{4}\zeta(6).  
\end{align*}

\noindent
Huard's expansion \eqref{huard-1} gives
\begin{align*}
T(2,1,3) & = 2 T(1,0,5)+T(2,0,4),  \\
T(2,2,2) & = 4 T(1,0,5)+2 T(2,0,4),  \\
T(3,1,2) & = 2 T(1,0,5)+T(2,0,4)+T(3,0,3), \\
T(3,2,1) & = 6 T(1,0,5)+3 T(2,0,4)+T(3,0,3).
\end{align*}

\noindent
Euler's identity \eqref{euler1},
\begin{equation*}
T(2,0,4) + T(4,0,2) = \zeta(2) \zeta(4) - \zeta(6)  = \tfrac{3}{4} \zeta(6), 
\end{equation*}
\noindent
allows us to express $T(2,0,4)$ in terms of $T(4,0,2)$. This, together
with the explicit evaluations of $T(1,0,5)$ and $T(3,0,3)$, yields
\begin{align*}
T(2,0,4) & = \tfrac{3}{4} \zeta(6) - T(4,0,2),\\
T(2,1,3) & = \tfrac{9}{4} \zeta(6) - \zeta^{2}(3) - T(4,0,2),  \\
T(2,2,2) & = \tfrac{9}{2} \zeta(6) - 2\zeta^{2}(3) - 2T(4,0,2),  \\
T(3,1,2) & = \tfrac{5}{2} \zeta(6) - \tfrac{1}{2}\zeta^{2}(3) - 2T(4,0,2),  \\
T(3,2,1) & = 7 \zeta(6) - \tfrac{5}{2}\zeta^{2}(3) - 4T(4,0,2). 
\end{align*}

\noindent
Finally, Huard's expansion \eqref{huard-1} applied to the Tornheim
sum $T(4,2,0)=\frac{7}{4}\zeta(6)$ produces another identity, 
\begin{equation*}
8 T(1,0,5)+4 T(2,0,4)+2 T(3,0,3)+T(4,0,2) = \tfrac{7}{4}\zeta(6),
\end{equation*}
which permits
to solve for the as yet undetermined value of $T(4,0,2)$:
\begin{equation}
T(4,0,2) = \tfrac{25}{12} \zeta(6) - \zeta^{2}(3).
\end{equation}

\medskip

\noindent
This last result produces the explicit evaluation of all Tornheim sums of 
weight $6$:

\begin{align}
T(1,0,5) &  =  \tfrac{3}{4} \zeta(6) -\tfrac{1}{2} \zeta^{2}(3), \quad &
T(1,1,4) & =  \tfrac{3}{2} \zeta(6) - \zeta^{2}(3), \nonumber  \\
T(2,0,4) & =  -\tfrac{4}{3} \zeta(6) +  \zeta^{2}(3), 
\quad &
T(2,1,3) & =   \tfrac{1}{6} \zeta(6),
 \nonumber \\
T(2,2,2) & =   \tfrac{1}{3} \zeta(6),
\quad &
T(3,0,3) & =  - \tfrac{1}{2} \zeta(6) + \tfrac{1}{2} \zeta^{2}(3),  \nonumber \\
T(3,1,2) & =  -\tfrac{1}{3} \zeta(6) + \tfrac{1}{2} \zeta^{2}(3),  
  \quad &
T(3,2,1) & =  \tfrac{1}{2} \zeta^{2}(3),
\nonumber \\
T(3,3,0) & = \zeta^{2}(3),  
 \quad &
T(4,0,2) & =  \tfrac{25}{12} \zeta(6) - \zeta^{2}(3),  
\nonumber \\
T(4,1,1) & =  \tfrac{7}{4} \zeta(6) - \tfrac{1}{2} \zeta^{2}(3),
\quad &
T(4,2,0) & =  \tfrac{7}{4} \zeta(6). \nonumber 
\end{align}

\medskip

\noindent
All the  Tornheim sums of weight $6$ have been evaluated. 

\medskip

\begin{Note}
The problem of whether these sums are completely reduced is now equivalent to 
whether $\zeta^{2}(3)$ and $\zeta(6) = \pi^{6}/945$ are 
rationally related. It is conjectured
that $\zeta(3)/\pi^{3}$ is a transcendental number\footnote{The authors 
wish to thank W. Zudilin for this information.}.
\end{Note}

\bigskip

\begin{center}
{\bf Weight $8$ }
\end{center}

\medskip

The reduction algorithm described above begins by applying Huard's reduction 
procedure to express every sum $T(m,k,n)$, with $N = m+k+n=8$, in terms of 
the $N-2 = 6$ sums 
\begin{equation}
\{ T(1,0,7), \, T(2,0,6), \, T(3,0,5), \, T(4,0,4), \, T(5,0,3), \, T(6,0,2) 
\, \} 
\end{equation}
\noindent
with $k=0$. For example, 
\begin{equation}
T(5,2,1) = 10T(1,0,7) + 5T(2,0,6) + 3T(3,0,5) + 2T(4,0,4) + T(5,0,3). 
\end{equation}
\noindent
We use the shorthand notation $T(5,2,1) = [ 10, 5, 3, 2, 1, 0]$. The table 
below gives all the coefficients corresponding to the $15$ Tornheim sums of 
weight $8$ with $k \neq 0$. 
{\allowdisplaybreaks
\begin{align*}
T(1,1,6) & =  \left[ 2, 0, 0, 0, 0, 0 \right], &
T(2,1,5) & =  \left[ 2, 1, 0, 0, 0, 0 \right], \\ 
T(2,2,4) & =  \left[ 4, 2, 0, 0, 0, 0 \right], &
T(3,1,4) & =  \left[ 2, 1, 1, 0, 0, 0 \right], \\
T(3,2,3) & =  \left[ 6, 3, 1, 0, 0, 0 \right], & 
T(3,3,2) & =  \left[ 12, 6, 2, 0, 0, 0 \right], \\
T(4,1,3) & =  \left[ 2, 1, 1, 1, 0, 0 \right], &
T(4,2,2) & =  \left[ 8, 4, 2, 1, 0, 0 \right], \\
T(4,3,1) & =  \left[ 20, 10, 4, 1, 0, 0 \right], &
T(4,4,0) & =  \left[ 40, 20, 8, 2, 0, 0 \right],\\  
T(5,1,2) & =  \left[ 2, 1, 1, 1, 1, 0 \right], & 
T(5,2,1) & =  \left[ 10, 5, 3, 2, 1, 0 \right], \\
T(5,3,0) & =  \left[ 30, 15, 7, 3, 1, 0 \right], & 
T(6,1,1) & =  \left[ 2, 1, 1, 1, 1, 1 \right], \\
T(6,2,0) & =  \left[ 12, 6, 4, 3, 2, 1 \right].\\
\stoli
\end{align*}
}

\medskip

Therefore every Tornheim sum $T(m,k,n)$ has been expressed in terms of the 
set 
\begin{equation}
\{ T(i,0,N-i): \, 1 \leq i \leq N-2 \}.
\end{equation}
\noindent
The value $T(1,0,N-1)$ is given in (\ref{tor-new3-1}) and $T(N/2,0,N/2)$ 
appears in (\ref{euler2-0}). Moreover, Euler's relation (\ref{euler1}) 
reduces the number of unknown Tornheim sums to $N/2 - 2$. In the case $N=8$
the two unknowns are $T(6,0,2)$ and $T(5,0,3)$. Among the $15$ sums discussed 
above, there are three with last index equal to $0$, namely $T(6,2,0), \, 
T(5,3,0)$ and $T(4,4,0)$. Each one of them produces an equation in the 
unknowns $T(6,0,2)$ and $T(5,0,3)$ coming from the evaluation 
\begin{equation}
T(m,k,0) = \zeta(m) \zeta(k).
\end{equation}
\noindent
For instance, the case $T(4,4,0)$ gives 
\begin{equation}
5T(6,0,2) + 2T(5,0,3) = \frac{163}{12} \zeta(8) - 8 \zeta(3) \zeta(5). 
\end{equation}
\noindent
This is the same relation among these sums obtained by Huard in 
(\ref{huard-7}). Unfortunately, the sums 
$T(6,2,0)$ and $T(5,3,0)$ yield the same relation, so 
we are unable to produce an analytic expression for all Tornheim sums of 
weight $8$, free of an unevaluated integral. \\

We conclude that every Tornheim sum 
of weight $8$ is a rational linear combination of 
the set 
\begin{equation}
G_{8} := \{ \zeta(8), \, \zeta(3) \zeta(5), \, T(6,0,2) \}. 
\end{equation}
\noindent
The table shows the corresponding coefficients: \\
\begin{align*}
T(1,0,7) & =  \left[ \tfrac{5}{4}, -1, 0 \right] &
T(1,1,6) & =  \left[ \tfrac{5}{2}, -2, 0 \right]  \nonumber \\
T(2,0,6) & =  \left[ \tfrac{2}{3}, 0, -1  \right]  &
T(2,1,5) & =  \left[ \tfrac{19}{6}, -2, -1  \right]
  \nonumber \\
T(2,2,4) & =  \left[ \tfrac{19}{3}, -4, -2   \right] &
T(3,0,5) & =  \left[ -\tfrac{187}{24}, 5, \tfrac{5}{2}  \right]
  \nonumber \\
T(3,1,4) & =  \left[ -\tfrac{37}{8}, 3, \tfrac{3}{2}  \right] &
T(3,2,3) & =  \left[ \tfrac{41}{24}, -1, -\tfrac{1}{2} \right]
  \nonumber \\
T(3,3,2) & =  \left[ \tfrac{41}{12}, -2, -1  \right] &
T(4,0,4) & =  \left[ \tfrac{1}{12}, 0, 0 \right] \nonumber \\
T(4,1,3) & =  \left[ -\tfrac{109}{24}, 3, \tfrac{3}{2}  \right] &
T(4,2,2) & =  \left[ -\tfrac{17}{6}, 2, 1   \right]
  \nonumber \\
T(4,3,1) & =  \left[ \tfrac{7}{12}, 0, 0 \right] &
T(4,4,0) & =  \left[ \tfrac{7}{6}, 0, 0 \right] \nonumber \\
T(5,0,3) & =  \left[ \tfrac{163}{24}, -4, -\tfrac{5}{2}   \right] &
T(5,1,2) & =  \left[ \tfrac{9}{4}, -1, -1   \right]
  \nonumber \\
T(5,2,1) & =  \left[ -\tfrac{7}{12}, 1, 0  \right] &
T(5,3,0) & =  \left[ 0, 1, 0   \right]
  \nonumber \\
T(6,0,2) & =  \left[ 0, 0, 1  \right] &
T(6,1,1) & =  \left[ \tfrac{9}{4}, -1, 0    \right] \nonumber \\
T(6,2,0) & =  \left[ \tfrac{5}{3}, 0, 0    \right]. & & 
  \nonumber 
\end{align*}
\medskip

The results derived in this paper give us the remaining unevaluated
Tornheim sum $T(6,0,2)$ in terms of the integral $Y_{2,6}^{*}$ as
\begin{equation}
T(6,0,2) = \tfrac{7}{6} \zeta(8) - 6\zeta(3)\zeta(5) - Y_{2,6}^*.
\label{nice-1a}
\end{equation}

Therefore, the generating set for Tornheim sums of weight 8 can also
be taken as
\begin{equation}
G_{8}^{*} := \{ \zeta(8), \, \zeta(3) \zeta(5), \, Y_{2,6}^* \}. 
\end{equation}

\bigskip

\begin{center}
{\bf Weight $10$ }
\end{center}

For weight 10, the algorithm follows step by step the previous case.
We find that all Tornheim sums in $\mathcal{Z}_{10}^{0}$ are generated
by the set
\begin{equation}
G_{10}:= \{ \zeta(10), \, \zeta^{2}(5), \, 
\zeta(3) \zeta(7), T(8,0,2) \}. 
\end{equation}
For example, 
\begin{equation*}
T(3,2,5) = -\tfrac{103}{40}\zeta(10) + \zeta^{2}(5) + \zeta(3)\zeta(7)
+\tfrac{1}{2} T(8,0,2).
\end{equation*}
\noindent
According to Theorem \ref{really-final}, the remaining unevaluated Tornheim
sum $T(8,0,2)$ can be expressed in terms of the integral $Y_{2,8}^{*}$ as
\begin{equation}
T(8,0,2) = \tfrac{23 }{20}\zeta (10)
   -8 \zeta (3) \zeta (7)-4 \zeta (5)^2 + Y_{2,8}^*.
\label{nice-2a}
\end{equation}
\noindent
In particular, we may also use the generating set 
\begin{equation}
G_{10}^{*}:= \{ \zeta(10), \, \zeta^{2}(5), \, 
\zeta(3) \zeta(7), Y_{2,8}^{*} \}. 
\end{equation}

\bigskip

\begin{center}
{\bf Generating set for Tornheim sums of even weight}
\end{center}

The same algorithm described above can be used to produce a generating set 
for the Tornheim sums of even weight $N$.
Except for $T(0,0,N)=\zeta(N-1)-\zeta(N)$, every such sum
is a rational linear combination of the elements of the set
\begin{equation} 
\left\{ \zeta(N), \, \zeta(j) \zeta(N-j): \, j \text{ odd }, 
\, 3 \leq j \leq 2\lfloor\frac{N-1}{4}\rfloor + 1 \, \right\}
\end{equation}
\noindent
and, for $N\ge 8$, one must also include the collection of integrals
\begin{equation}
\left\{ Y_{2r,N-2r}^*: \, 1 \leq r \leq 
\left\lfloor \frac{N-2}{6} \right\rfloor \right\},
\end{equation}
where
where $Y_{2r,N-2r}^{*}$ is defined in (\ref{def-Y-new}). 

\medskip 

The smallest weight for which one requires two basis Tornheim sums is $14$. 
In this case the generating set is
\begin{equation}
G_{14}:= \{ \zeta(14), \, \zeta^{2}(7), \, 
\zeta(5) \zeta(9), \zeta(3) \zeta(11), T(12,0,2), T(10,0,4) \}. 
\end{equation}
The evaluation of $T(12,0,2)$ and $T(10,0,4)$ according to Theorem
\ref{really-final} gives
\begin{align*}
T(12,0,2) &= \tfrac{481 }{420}\zeta (14)-12 \zeta (3) \zeta (11)-12 \zeta (5) \zeta
   (9)-6 \zeta (7)^2 + Y_{2,12}^*,\cr
T(10,0,4) &= \tfrac{7 }{12}\zeta (14)-120 \zeta (3) \zeta (11)-60 \zeta (5) \zeta
   (9)-20 \zeta (7)^2 + Y_{4,10}^*,
\end{align*}
with
\begin{align*}
Y_{2,12}^* = Y_{2,12} + 12 \zeta(13) \log 2 \pi,\cr
Y_{4,10}^* = Y_{4,10} + 220 \zeta(13) \log 2 \pi.
\end{align*}
\noindent
In particular, we may also use the generating set 
\begin{equation}
G_{14}^{*}:= \{ \zeta(14), \, \zeta^{2}(7), \, 
\zeta(5) \zeta(9), \zeta(3) \zeta(11), Y_{2,12}^{*}, Y_{4,10}^{*} \}. 
\end{equation}


\begin{Note}
In \cite{borwein95} the reader will find the sums 
\begin{equation}
\sigma_{h}(s,t) := \sum_{n=1}^{\infty} \sum_{k=1}^{n-1} 
\frac{1}{k^{s}} \frac{1}{n^{t}},
\end{equation}
\noindent
that can be expressed as 
\begin{equation}
\sigma_{h}(s,t) = \zeta(s) \zeta(t) - \zeta(s+t) - T(t,0,s). 
\end{equation}
\noindent
The authors analyze  a system of equations for the sums $\sigma_{h}(s,t)$ 
with the weight $w:= s+t$ fixed. For $w$ odd, the system has full rank and 
they obtain Huard's expression for the Tornheim sums. In the case 
$w = 2n$ even, they establish that the dimension of the null space is 
$\lfloor(n-1)/3 \rfloor$, so every sum can be expressed in terms of 
this number of basis elements. In our 
case, $N = n/2$, thus the expected number of 
generators for all Tornheim sums of weight $N$ is at most $\lfloor(N-2)/6 
\rfloor$. The  fact that the set
\begin{equation}
\left\{ \, T(N-2r,0,2r): \, 1 \leq r \leq 
\lfloor (N-2)/6 \rfloor \right\}
\end{equation}
\noindent
can be used to generate all Tornheim sums (aside from the usual product of 
zeta values) will follow from a careful analysis of the identities generated 
by the relations $T(m,n,0) = \zeta(m) \zeta(n)$. We leave the details for the 
ambitious reader. The fact that these sums 
are linearly independent is beyond our reach. 
\end{Note}

\bigskip

\section{Conclusions} \label{S:conclusions}

We have discussed an algorithm that evaluates all the Tornheim sums 
\begin{equation}
T(m,k,n)  : =  \sum_{r=1}^{\infty} \sum_{s=1}^{\infty} \frac{1}{r^{m}
\, s^{k} \, (r+s)^{n} },
\label{seriestz-0}
\end{equation}
\noindent
of a given even weight $N := m+k+n$, as rational linear combinations of
the value $\zeta(N)$,  
products of zeta values $\zeta(j) \zeta(N-j)$ with $j$ odd in the range 
$3 \leq j \leq 2\lfloor\frac{N-1}{4}\rfloor + 1$ and the integrals
\begin{equation}
\left\{ Y_{2r,N-2r}^*: \, 1 \leq r \leq 
\left\lfloor \frac{N-2}{6} \right\rfloor \right\},
\end{equation}
\noindent
where 
\begin{multline}
Y_{m,n}^{*} := \frac{2(2\pi)^{m+n-2}}{m!(n-2)!} 
\sum_{j=0}^{m} (-1)^{j} \binom{m}{j} X_{j,m+n-2-j} \label{def-Y-0} \cr
+ (-1)^{N/2-1} \binom{N-2}{m-1} \zeta(N-1) \log 2 \pi,
\end{multline}
\noindent
and the integral $X_{k,l}$ is defined by 
\begin{equation}
X_{k,l} := (-1)^{\lfloor{l/2 \rfloor}} 
\frac{l!}{(2\pi)^l}\ione  \log \Gamma(q) B_{k}(q) 
\Cl_{l+1} (2 \pi q)  \, dq. 
\label{def-X-0}
\end{equation}
\noindent
Here $B_{k}$ is the Bernoulli polynomial and $\Cl$ is the Clausen function.  \\

All the Tornheim sums of a given even weight $N$ can be expressed in terms of
zeta values and a reduced number of basis sums of the type $T(N-2r,0,2r)$, with
$r=1,\ldots,\left\lfloor \frac{N-2}{6} \right\rfloor$. These sums, in turn,
can themselves be expressed in terms of zeta values and the integral
$Y_{2r,N-2r}^*$, according to Theorem \ref{really-final}:
\begin{multline}
T(N-2r,0,2r) = (-1)^{N/2-1} Y_{2r,N-2r}^* + \zeta(2r)\zeta(N-2r) - \frac{1}{2}\zeta(N)\cr
-\sum _{j=1}^{N/2-2} \binom{N-2-2j}{2r-1}\zeta (2j+1) \zeta (N-1-2j).
\end{multline}

\medskip

Our results may perhaps be used to develop fast numerical codes to compute
even weight Tornheim sums to high accuracy.
Since the whole family of Tornheim sums of a given weight can be expressed
in terms of zeta values and a small basis of Tornheim sums, it is enough to
compute the basis sums to the required accuracy. This will involve the 
numerical calculation of the $Y$ integrals. 

For example, of the 46 admissible triples at weight $N=12$, 30 give Tornheim
sums that depend on the value of the single $T(10,0,2)$. This sum is given by
\begin{equation}
T(10,0,2) = - Y_{2,10}^{*} - 10 \zeta(5) \zeta(7) - 
10 \zeta(3) \zeta(9) + \tfrac{792}{691} \zeta(12). 
\end{equation}
\noindent
A 30-digit precision
Mathematica calculation of the integral $Y_{2,10}^{*}$ gives 
\begin{equation}
T(10,0,2) = 0.645\,324\,784\,017\,496\,594\,071\,783\,081\,476\ldots.
\end{equation}

\medskip

\no
{\bf Acknowledgments}. The first author would like to thank the
Department of Mathematics at Tulane University for its hospitality.
The second author acknowledges the partial support of NSF-DMS 00409968. The
authors wish to thank D. Broadhurst for comments on an earlier version 
of this paper.


\bigskip

\begin{thebibliography}{10}

\bibitem{borwein95}
D.~Borwein, J.~M. Borwein, and R.~Girgensohn.
\newblock Explicit evaluation of {E}uler sums.
\newblock {\em Proc. Edin. Math. Soc.}, 38:277--294, 1995.

\bibitem{borw2}
J.~M. Borwein, D.~H. Bailey, and R.~Girgensohn.
\newblock {\em Experimentation in Mathematics: Computational Paths to
  Discovery}.
\newblock A. K. Peters, 1st edition, 2004.

\bibitem{borwein-bradley1}
J.~M. Borwein and D.~Bradley.
\newblock Thirty-two {G}oldbach variations.
\newblock {\em International Journal of Number Theory}, 2:65--103, 2006.

\bibitem{espmoll1}
O.~Espinosa and V.~Moll.
\newblock On some definite integrals involving the {H}urwitz zeta function.
  {P}art 1.
\newblock {\em The {R}amanujan {J}ournal}, 6:159--188, 2002.

\bibitem{espmoll2}
O.~Espinosa and V.~Moll.
\newblock On some definite integrals involving the {H}urwitz zeta function.
  {P}art 2.
\newblock {\em The {R}amanujan {J}ournal}, 6:449--468, 2002.

\bibitem{espmoll4}
O.~Espinosa and V.~Moll.
\newblock A generalized polygamma function.
\newblock {\em Integral {T}ransforms and {S}pecial {F}unctions}, 15:101--115,
  2004.

\bibitem{espmoll3}
O.~Espinosa and V.~Moll.
\newblock The evaluation of {T}ornheim double sums. {P}art 1.
\newblock {\em Journal of {N}umber {T}heory}, 116:200--229, 2006.

\bibitem{gr}
I.~S. Gradshteyn and I.~M. Ryzhik.
\newblock {\em Table of {I}ntegrals, {S}eries, and {P}roducts}.
\newblock Edited by A. Jeffrey and D. Zwillinger. Academic Press, New York, 7th
  edition, 2007.

\bibitem{granville1}
A.~Granville.
\newblock A decomposition of {R}iemann's zeta function.
\newblock {\em London Mathematical Society}, 247:95--101, 1997.

\bibitem{huard}
J.~Huard, K.~Williams, and N.~Zhang.
\newblock On {T}ornheim's double series.
\newblock {\em Acta Arith.}, 75:105--117, 1996.

\bibitem{tornheim1}
L.~Tornheim.
\newblock Harmonic double series.
\newblock {\em Amer. J. Math.}, pages 303--314, 1950.

\end{thebibliography}
\end{document}